\documentclass[a4paper,10pt]{article}
\usepackage[paper=a4paper, left=3.5cm, right=3.5cm, top=3cm, bottom=3cm]{geometry}

\usepackage[utf8]{inputenc}
\usepackage[english]{babel}

\usepackage{todonotes}
\usepackage{algorithm,algpseudocode}
\usepackage{amsmath,amssymb,amsthm}
\usepackage[normalem]{ulem} 
\usepackage{cancel} 

\graphicspath{{./}{figs/}}

\newcommand{\R}{\mathbb R}

\newtheorem{dfntn}{Definition}[section]
\newtheorem{thrm}{Theorem}[section]
\newtheorem{rmk}{Remark}[section]

\newtheorem{lmm}{Lemma}[section]

\newcommand{\ic}{0 }

 \title{Time adaptivity in model predictive control}
 \author{Alessandro Alla$^1$, Carmen Gr\"a\ss{}le$^2$, Michael Hinze$^3$}
\date{\small $^1$ Department of Mathematics, PUC-Rio, Rio de Janeiro, Brasil, alla@mat.puc-rio.br\\
$^2$ Max Planck Institute for Dynamics of Complex Technical Systems, Magdeburg, Germany\\
$^3$ Mathematical Institute, University of Koblenz-Landau, Koblenz, Germany \footnote{Michael Hinze acknowledges support of the German ministry for Education and Research within PASIROM under the grant 05M18GUA.}}

\usepackage{hyperref}

\begin{document}

\maketitle

\begin{abstract} 
The core of the Model Predictive Control (MPC) method in every step of the algorithm consists in solving a time-dependent optimization problem on the prediction horizon of the MPC algorithm, and then to apply a portion of the optimal control over the application horizon to obtain the new state. To solve this problem efficiently, we propose a time-adaptive residual a-posteriori error control concept based on the optimality system of this optimal control problem. This approach not only delivers a tailored time discretization of the the prediction horizon, but also suggests a tailored length of the application horizon for the current MPC step. We apply this concept for systems governed by linear parabolic PDEs and present several numerical examples which demonstrate the performance and the robustness of our adaptive MPC control concept.

\end{abstract}

%

\section{Introduction}\label{sec:intro}

In this article we consider Model Predictive Control (MPC) for systems governed by linear parabolic PDEs. This approach is also known as {\em moving horizon control} or {\em receding horizon control}, where we refer to the (seminal) monographs \cite {GruP17,RawMD09} for a comprehensive presentation of this method. The core of the method for every MPC step at time $t$ consists in solving a parabolic PDE constrained optimization problem on the prediction horizon $[t,t+\bar{T}]$, where $\bar{T}>0$. To solve this problem efficiently we propose a time-adaptive residual based a-posteriori error control concept for the elliptic space-time reformulation of the optimality system of the PDE constrained optimization problem. The contribution of our paper and the novelty of our approach is two-fold; \begin{itemize}
    \item it delivers a tailored time discretization of the prediction horizon $[t,t+\bar{T}]$ using residual based a-posteriori error control concepts, and
    \item also suggests a tailored length $\tau \le \bar{T}$ of the application horizon $[t,t+\tau]\subseteq [t,t+\bar T]$ for the current MPC step. 
\end{itemize} 
Our time-adaptive MPC algorithm works as follows, where the details of its formulation are given in Section \ref{sec1}.
 \begin{algorithm}[H]
\caption{Model predictive control (MPC) with adaptive time grids }
\label{Alg:NMPC_on}
\begin{algorithmic}[1]
\Require Number of time instances in each subinterval $N$, prediction horizon $\bar{T}$, initial condition $y_\circ$, desired state $y_d$, source $f$, space domain $\Omega$, constants $\nu,\mu,\alpha$. 
\State Set $t_0=0, y_0=y_\circ$.
\For{$i =0,1,2,\ldots$}
\State  Compute an adaptive grid $\{t_i^{j}\}_{j=1}^{N}$ in $[t_i,\bar t_i:=t_i+\bar{T}]$ using \eqref{est-thm31mu}, where $t_i^1:= t_i$ and $t_i^{N}:= \bar t_i$.
\State Compute the optimal control over $[t_i,\bar t_i]$ solving
\begin{equation}\label{MPC:FinHor_on}
u^N:= \arg\min_{ u\in L^2((t_i,\bar t_i);\Omega)} \hat{J}^N(u,t_i,y_i).
\end{equation}
\State Define the length of the application horizon $\tau_i:= t_i^2-t_i^1$.
\State Define the MPC feedback value $\phi^N(y_{[u^N,t_i,y_i]}(t))=u^N(t)$ for  $t\in(t_i,t_i+\tau_i]$.
\State Compute the associated state $y^N=y_{[\phi^N,t_i,y_i]}(t)$ by solving \eqref{heatdepletion}  on $(t_i,t_i+\tau_i]$.
\State Set $t_{i+1}:=t_i+\tau_i$, $y_{i+1}:=y^N(t_{i+1})$, 
$i \leftarrow i+1$.
\EndFor
\end{algorithmic}
\end{algorithm}

Our adaptive concept is implemented with the first statement of the {\bf For-loop} in Algorithm \ref{Alg:NMPC_on} and works as follows: in a first step, we rewrite the optimality conditions of the MPC optimization problem \eqref{MPC:FinHor_on} as a second order in time and fourth order in space elliptic equation for the state variable, to which we then apply classical concepts from residual based a-posteriori error control for the time variable. This allows to construct a time grid for the state which is related to the \textit{optimal} state solution and, at the same time, delivers the length $\tau$ of the application horizon in the current MPC step. The idea is based on \cite{GonHZ12}, and now is transferred for a mixed formulation, where the a-posteriori error estimate is obtained from a semi-time discrete mixed form. For the fast computation of the adaptive time grid we use a coarse spatial discretization, where we assume that the structure of the temporal grid is not sensitive against changes in the spatial resolution. This is verified heuristically by numerical examples in e.g.\ \cite{AllGH16,AllGH18}. In a second step the resulting time grid is used for the numerical solution of the MPC optimization problem \eqref{MPC:FinHor_on}, where we use the control and the state obtained from the computation of the adaptive grid as initialization for the  solution procedure. Finally, the state is updated on the application horizon $[t_i,t_i+\tau_i]$ through the solution of the parabolic equation \eqref{heatdepletion} with the optimal control $u^N$.

Let us briefly comment on related literature. Since there is a vast amount of books and papers on MPC we here concentrate on contributions related to adaptivity in MPC. In \cite{GruSS19, Kre18} the authors took advantage of the structure of the problem using Lyapunov functions and/or the turnpike property to construct adaptive grids for the MPC optimal control problem. The turnpike property (see e.g.\ \cite{Zas06}) is often a key tool to prove asymptotic stability of the MPC method and to find the minimal prediction horizon (see e.g.\ \cite{BreP20,GruPSetal09,GruSS20,KunP19}). Our ideas are related to \cite{GSSnew}, where a goal-oriented  adaptive approach for the MPC optimal control problem is proposed. This paper appeared while we were editing the first version of our manuscript. However, the a-posteriori concepts proposed there differ from our approach which relies on residual based a-posteriori error analysis for the elliptic space-time reformulation of the optimality systems appearing in every step of the MPC algorithm. Our method is tailored directly by the optimal state.

The outline of this paper is as follows. In Section~\ref{sec1}, we present the optimal control problem within the MPC framework and recall the basic idea of the MPC method. Further, we state the optimality conditions for the MPC subproblem. In Section~\ref{sec2} we describe the reformulation of the optimality system to a second order in time and fourth order in space elliptic equation as well as a mixed variational form. Further, we derive an a-posteriori error estimate for a semi-time discrete form. In Section~\ref{secadaptivity}, we propose the novel time-adaptive scheme in MPC. Finally, numerical tests are discussed in Section~\ref{secnumerics} and conclusions are made in Section~\ref{secconcl}.\\

\section{Optimal control setting within the MPC framework}\label{sec1}
\subsection{Preliminaries}
Let $\Omega \subset \mathbb{R}^n, n \in \{1,2,3\}$ be an open and bounded domain with Lipschitz boundary $\partial \Omega.$ The Lebesgue space of square integrable functions is denoted by $L^2(\Omega)$  with inner product $(u,v)_{L^2(\Omega)}:= \int_\Omega uv dx$ and norm $\|u\|_{L^2(\Omega)}:=(\int_\Omega |u(x)|^2 dx)^{1/2}$ for $u,v \in L^2(\Omega)$. Further, let $H^k(\Omega)$ defined by 
$$H^k(\Omega) := \{u\in L^2(\Omega): u \text{ has }\text{weak }\text{derivatives } D^{\beta}u \in L^2(\Omega) \text{ for all } |\beta | \leq k\}$$ 
with $k \in \mathbb{N}_0$ and equipped with the norm $\|u\|_{H^{k}(\Omega)}  := (\sum_{|\beta| \leq k} \|D^\beta u \|_{L^2(\Omega)}^2)^{1/2}$  and
$$ H_0^k(\Omega):= \{ u \in H^k(\Omega): D^\beta u = 0 \text{ on } \partial \Omega \text{ in the sense of traces } (|\beta| \leq k-1)\}. $$
\noindent We use the notation $H^{-1}(\Omega)$ for the dual space of $H_0^1(\Omega)$ and denote $\langle \cdot , \cdot \rangle_{H^{-1}(\Omega),H_0^1(\Omega)}$ as the duality pairing of $H^{-1}(\Omega)$ with $H_0^1(\Omega)$. By $| \cdot |_{H^1(\Omega)}$ we denote the $H^1$-seminorm given by $|u|_{H^1(\Omega)} = \|\nabla u \|_{L^2(\Omega)}$ for $u\in H_0^1(\Omega)$. We recall that the Poincar\'{e} constant is given by the smallest number $c_p>0$ such that the Poincar\'{e} inequality 
$$\|u\|_{L^2(\Omega)} \leq c_p \|\nabla u \|_{L^2(\Omega)}, \; \forall u \in H_0^1(\Omega)$$ 
is fulfilled. Thus, $|.|_{H^1(\Omega)}$ is a norm on $H_0^1(\Omega)$ equivalent to the norm $\|.\|_{H^1(\Omega)}$. For a given Banach space $X$ and a given time $T>0$, we denote by $L^2((\ic,T);X)$ the space of  measurable square integrable abstract functions with norm $\|u\|_{L^2((\ic,{T});X)} := (\int_{\ic}^{T} \|u(t)\|_X^2 dt )^{1/2}$. We define 
$$W((\ic,T);H_0^1(\Omega)):=\{v\in L^2((\ic,T);H_0^1(\Omega)), v_t \in L^2((\ic,T);H^{-1}(\Omega))\}.$$ 
Note that for a given function $g$ in space-time, we use the short hand notation $g(t)$ to indicate the time dependency and drop the space argument.

\subsection{Model predictive control}\label{secmpc}

In this section we specify our MPC setting of Algorithm \ref{Alg:NMPC_on}. At time $t_0$ we initialize our MPC algorithm and for convenience use a fixed length $\bar T > 0$ for the prediction horizon. At time instance $t_i \ge t_0$ $(i\in \mathbb{N})$ this horizon is denoted by $[t_i,\bar t_i]$ with $\bar t_i := t_i+\bar T$. We denote with $\tau_i$ the length of the application horizon at time instance $t_i$, so that $[t_i,t_i+\tau_i] \subseteq [t_i,\bar t_i]$. The adaptive time grid at time instance $t_i$ is denoted by $\{t_i^j\}_{j=1}^{N}$, where we set $t_i^1:=t_i$ and $t_i^N:=\bar t_i$. Form here onwards we use $t_i^N$ instead of $\bar t_i$ to denote the final time in the prediction horizon. We note that the value of $\tau_i = t_i^2-t_i^1$ may change with every time instance $t_i$ due to our time-adaptive concept.

The reduced cost functional over the domain $[t_i, t_i^N] \times \Omega$ which is considered at the $i$-th time instance of the MPC algorithm for $i=0,1,2\dots,$ is given by
\begin{equation}\label{mpccost}
\hat{J}^N(u,t_i,y_i):=\int_{t_i}^{t_i^N} \ell (y_{[u,t_i,y_i]}(t),u(t))\,dt,
\end{equation}
where the function $\ell$ in our applications is given by
\begin{equation}\label{costJ}
\ell(y(t),u(t)):= \frac{1}{2} \|y(t) - y_d(t)\|_{L^2(\Omega)}^2 + \frac{\alpha}{2}\|u(t)\|_{L^2(\Omega)}^2.
\end{equation}
Here $y_d \in L^2((t_i,t_i^N);\Omega)$ denotes the desired state and $\alpha > 0$ the prescribed regularization parameter. To anticipate discussions we note that also other cost functionals could be considered.
The governing dynamics for the state $y\equiv y_{[u,t_i,y_i]}$ is given by the linear parabolic partial differential equation
\begin{equation}\label{heat}
\left\{
\begin{array}{rcll}
y_t-\nu\Delta y & = & f+u &\text{ in } (t_i,t_i^N]\times\Omega,\\
y &= & 0 &\text{ on } (t_i,t_i^N]\times \partial \Omega,\\
y(t_i) & = & y_i & \text{ in } \Omega,
\end{array}
\right.
\end{equation}
where $\nu > 0$ is a given constant, $f$ is a given source term and $y_i$ is the given initial state which is obtained from the preceding MPC step. The function $u$ will act as the control.
The weak form of \eqref{heat} reads: for given $f \in L^2((t_i,t_i^N);\Omega)$, $y_i \in L^2(\Omega)$ and $u\in L^2((t_i,t_i^N);\Omega)$, find a state $y \in W((t_i,t_i^N);H_0^1(\Omega))$ satisfying $y(t_i) = y_i$ such that
\begin{equation}\label{heat_weak}
 \langle y_t(t),v \rangle_{H^{-1}(\Omega),H_0^1(\Omega)} + \nu   \int_\Omega \nabla y(t) \cdot \nabla v dx  = \int_\Omega ( f(t)+u(t))v dx
\end{equation}
holds for all $v \in H_0^1(\Omega)$ and almost everywhere in $(t_i,t_i^N]$. It is clear that \eqref{heat_weak} admits a unique weak solution, see e.g.\ \cite[\textsection 7.1.2, Theorems 3 and 4]{Eva10}. It therefore is meaningful to consider the state $y$ as a function of the control $u$, so that the cost functional in \eqref{mpccost} in fact only depends on the control as independent variable.\\

Then, the open loop control problem in the $i-$th  optimization instance of the MPC method is given by
\begin{equation}\label{mpcocpreduced}
 \min_{u \in L^2((t_i,t_i^N);\Omega)} \hat{J}^N(u,t_i,y_i).
\end{equation}
It forms the core of every MPC step. In the next section we develop a time-adaptive concept for its numerical approximation.

\subsection{Optimal control problem}\label{sec:ocp}

In this section, we investigate the distributed optimal control problem which we consider in each level of the MPC framework. 
To ease the notation here we will consider a general finite horizon $[0,T]$ instead of $[t_i,t_i^N]$.
It is clear that in the setting of the previous section the optimal control problem \eqref{mpcocpreduced} admits a unique solution $u\in L^2((0,T);L^2(\Omega))$. Moreover, there exists a unique adjoint state $p\in W((0,T); H^1_0(\Omega))$ which together with $u$ and the state $y\in W((0,T); H^1_0(\Omega))$ satisfies the optimality system consisting of the state equation
\begin{equation}\label{heat1}
\left\{
\begin{array}{rcll}
y_t-\nu\Delta y & = & f+u &\text{ in } (\ic,T]\times\Omega,\\
y &= & 0 &\text{ on } [\ic,T] \times \partial \Omega,\\
y(\ic) & = & y_0 & \text{ in } \Omega,
\end{array}
\right.
\end{equation}
the adjoint equation
\begin{equation}\label{adj}
\left\{
\begin{array}{rcll}
-p_t-\nu\Delta p  & = & y-y_d &\text{ in } [\ic,T)\times\Omega,\\
p & = & 0 &\text{ on } [\ic,T] \times \partial \Omega,\\
p(T) & = & 0 &\text{ in } \Omega,
\end{array}
\right.
\end{equation}
and the optimality condition
\begin{equation}\label{opt_con}
\alpha u + p=0\quad \mbox{ in } [\ic,T]\times\Omega.
\end{equation}
\begin{rmk}
We note that it is possible to consider control constraints, state constraints and control operators mapping abstract controls to feasible right hand sides in \eqref{mpcocpreduced}, see Section~\ref{constraints} for a discussion.
\end{rmk}
In the next section we rewrite the optimality system as an elliptic boundary value problem in space-time and exploit its elliptic structure to provide adaptive concepts for its solution. For this purpose we need the following higher regularity results for the weak solutions of $y$ of \eqref{heat1} and $p$ of \eqref{adj}, respectively.
 \begin{lmm}[Higher regularity \cite{Eva10}]\label{lemma:regularity}
 (i) Let $y_0 \in H_0^1(\Omega)$ and let $f$, $u$, $y_d \in L^2((\ic,T);\Omega)$. Then, according to \cite[\textsection 7.1.3. Theorem 5]{Eva10} the weak solution $y$ of \eqref{heat} and the weak solution $p$ of \eqref{adj} fulfill $y,p \in L^2((\ic,T);H^2(\Omega)) \cap L^\infty((\ic,T);H_0^1(\Omega))\cap  H^1((\ic,T);L^2(\Omega))$.\\
 (ii) Let $y_0 \in H_0^1(\Omega) \cap H^3(\Omega)$ and $f,u, y_d \in L^2((\ic,T);H^2(\Omega))\cap H^1((\ic,T);L^2(\Omega))$. Further, let the compatibility assumption $(u+f)(0)+\nu \Delta y_0  \in H_0^1(\Omega)$ hold true. Then according to \cite[\textsection 7.1.3. Theorem 6]{Eva10} the weak solution $y$ of \eqref{heat} and the weak solution $p$ of \eqref{adj} fulfill $y,p \in L^2((\ic,T);H^4(\Omega))\cap H^1((\ic,T);H^2(\Omega))\cap H^2((\ic,T);L^2(\Omega))$.
 \end{lmm}

\section{Reformulation of the optimality system and time adaptivity}\label{sec2}
\subsection{Reformulation of the optimality system}\label{sec2:subreform}
Following along the lines of \cite{GonHZ12}, we can reformulate the optimality system \eqref{heat1}-\eqref{adj}-\eqref{opt_con} as an elliptic equation of fourth order in space and second order in time involving only the state variable $y$. The adjoint state $p$ as well as the control $u$ are not present in this equation. In particular, it is a two-point boundary value problem in space-time given by
\begin{equation}\label{2ord}
\left\{
\begin{array}{rcll}
-y_{tt}+\nu^2\Delta^2 y  +\frac{1}{\alpha} y & = &\frac{1}{\alpha}y_d -f_t-\nu \Delta f & \text{ in } (\ic,T)\times\Omega,\\
y & = & 0 &\text{ on } [\ic,T]\times\partial\Omega,\\
\nu \Delta y & = & -f &\text{ on } [\ic,T]\times\partial\Omega,\\
\left(y_t-\nu \Delta y \right)(T) & = & f(T) &\text{ in }\Omega,\\
y(\ic) & = & y_0 &\text{ in } \Omega.
\end{array}
\right.
\end{equation}
We note that for $\nu=1$ and $f\equiv0$ this setting coincides with the setting considered in \cite{GonHZ12}. Under higher regularity assumptions on the data, the following theorem shows that the optimal state $y$ of \eqref{heat1}-\eqref{adj}-\eqref{opt_con} fulfills the elliptic equation \eqref{2ord} a.e.\ in space-time.

\begin{thrm}
 Let $(y,u) \in W((\ic,T);H_0^1(\Omega))\times L^2((\ic,T);\Omega)$ with associated adjoint $p\in W((\ic,T);H_0^1(\Omega))$ denote the unique weak solution to \eqref{heat1}-\eqref{adj}-\eqref{opt_con}. Further, let the assumptions of Lemma~\ref{lemma:regularity}(ii) be fulfilled. Then, $y$ satisfies \eqref{2ord} a.e.\ in space-time.
\end{thrm}
\begin{proof}
The proof follows along the lines of the proof of \cite[Theorem 2.7]{GonHZ12} and uses differentiation and insertion of the equations \eqref{heat1}-\eqref{adj}-\eqref{opt_con}. 
\end{proof}

\noindent Let us homogenize \eqref{2ord}. For this, let $g$ be a function which fulfills the boundary conditions as well as initial and end time conditions of \eqref{2ord} and is sufficiently smooth. Let $y$ satisfy \eqref{2ord}. We define $\tilde{y}:=y-g$ and arrive at
\begin{equation}\label{2ord_hom}
\left\{
\begin{array}{rcll}
-\tilde{y}_{tt}+\nu^2\Delta^2 \tilde{y} + \frac{1}{\alpha} \tilde{y} & = &\tilde{y}_d & \text{ in } (\ic,T)\times\Omega,\\
\tilde{y} & = & 0 &\text{ on } [\ic,T]\times\partial\Omega,\\
\nu \Delta \tilde{y} & = & 0 &\text{ on } [\ic,T]\times\partial\Omega,\\
\left(\tilde{y}_t-\nu \Delta \tilde{y} \right)(T) & = & 0 &\text{ in }\Omega,\\
\tilde{y}(\ic) & = & 0 &\text{ in } \Omega,
\end{array}
\right.
\end{equation}
where 
\begin{equation}\label{yd}
\tilde{y}_d:= \frac{1}{\alpha} y_d - f_t - \nu \Delta f  + g_{tt} - \nu^2 \Delta^2 g - \frac{1}{\alpha} g.
\end{equation}

\noindent Now, let us derive a weak formulation of \eqref{2ord_hom}. For this purpose we introduce the function space
$$H^{2,1}_0((\ic,T);\Omega):=\left\{v\in H^{2,1}((\ic,T);\Omega): v(\ic)=0 \mbox { in }\Omega\right\},$$
where 
$$H^{2,1}((\ic,T);\Omega):=L^2((\ic,T);H^2(\Omega)\cap H^1_0(\Omega))\cap H^1((\ic,T); L^2(\Omega)).$$ It is equipped with the norm $$\|v\|_{H^{2,1}((\ic,T);\Omega)}:= \left( \|v\|_{L^2((\ic,T);H^2(\Omega))}^2 + \|v\|_{H^1((\ic,T);L^2(\Omega))}^2 \right)^{1/2}.$$ We introduce the following symmetric bilinear form 
$$A:H^{2,1}_0((\ic,T);\Omega)\times H^{2,1}_0((\ic,T);\Omega)\rightarrow\R,$$
$$A(v_1,v_2):=\displaystyle\int_{\ic}^T \int_\Omega \left( (v_1)_t (v_2)_t + \nu^2 \Delta v_1 \Delta v_2  + \frac{1}{\alpha}  v_1 v_2 \right)dxdt + \displaystyle\int_\Omega  \nu\nabla v_1(T) \nabla v_2(T) dx,$$
and linear form
$$L:H^{2,1}_0((\ic,T);\Omega)\rightarrow\R, \quad L(v) := \int_0^T \int_\Omega \tilde{y}_d v \; dxdt$$
where $\tilde{y}_d$ is defined in \eqref{yd}.

\begin{dfntn}(Weak formulation)
 The weak formulation of equation \eqref{2ord_hom} is given by: find $\tilde{y} \in H_0^{2,1}((\ic,T);\Omega)$ which satisfies
\begin{equation}\label{weak_2ord}
A(\tilde{y},v)=L(v)\quad \forall v\in H^{2,1}_0((\ic,T);\Omega).
\end{equation}
\end{dfntn}

Existence of a solution to \eqref{weak_2ord} and its relation to a solution to \eqref{2ord} is shown in the following theorem.

\begin{thrm}\label{thm:existence}
 Let $y$ denote a solution to \eqref{2ord} and let $g$ be a function which fulfills the boundary, initial and end time conditions in \eqref{2ord} and is sufficiently smooth. Then, $\tilde{y}=y-g$ is a solution to \eqref{weak_2ord}. On the other hand, if $\tilde{y}$ is a solution to \eqref{weak_2ord} and the assumptions of Lemma~\ref{lemma:regularity}(ii) are fulfilled, then $y=\tilde{y}+g$ satisfies \eqref{2ord} a.e.\ in space-time.
\end{thrm}
\begin{proof}
Assume $y$ is a solution to \eqref{2ord}. By Green's formula and integration by parts it is straight forward to prove that $\tilde{y}=y-g$ satisfies \eqref{weak_2ord}. The other direction follows vice versa.
\end{proof}

In order to show equivalence of the optimal control problem \eqref{mpcocpreduced} over $(0,T)$ to the weak formulation of \eqref{2ord} it remains to prove uniqueness of a solution.

\begin{thrm}\label{thm:unique}
 The solution $y$ to \eqref{weak_2ord} is unique. 
\end{thrm}
\begin{proof}

The proof follows along the lines of the proof of \cite[Theorem 2.6]{GonHZ12} {\color{black}and uses Lax-Milgram Lemma (see e.g.\ \cite[\textsection 6.2.1, Theorem 1]{Eva10})}.
\end{proof}

\subsection{Mixed formulation}\label{secmixedform}
 In order to use piecewise linear, continuous finite elements for discretization and avoid the construction of finite element subspaces in $H^2(\Omega)$, we introduce an auxiliary variable $\tilde{w}:=-\nu \Delta \tilde{y}$. This allows to write \eqref{2ord_hom} as a coupled system in $\tilde{y}$ and $\tilde{w}$ as
\begin{equation}\label{2ord_decoupled}
 \left\{
 \begin{array}{rcll}
 -\tilde{y}_{tt}-\nu \Delta \tilde{w}  +\frac{1}{\alpha} \tilde{y} & = &\tilde{y}_d & \text{ in } (\ic,T)\times\Omega,\\
 \nu \Delta \tilde{y} + \tilde{w} & = & 0 &  \text{ in } (\ic,T)\times\Omega,\\
 \tilde{y} & = & 0 &\text{ on } [\ic,T]\times\partial\Omega,\\
  \tilde{w} & = & 0 &\text{ on } [\ic,T]\times\partial\Omega,\\
 \left(\tilde{y}_t-\nu \Delta \tilde{y} \right)(T) & = & 0 &\text{ in }\Omega,\\
 \tilde{y}(\ic) & = & 0 &\text{ in } \Omega.
 \end{array}
 \right.
 \end{equation}
 We introduce the function spaces $Y:=\{v \in H^1((\ic,T);H_0^1(\Omega)): v(\ic) = 0 \text{ in } \Omega \}$, $W:=L^2((\ic,T);H_0^1(\Omega))$ and the product space $X:=Y\times W$. Let us define the following bilinear form
 $$ A_M:X \times X \to \mathbb{R},$$
\begin{align*}
A_M((\tilde{y},\tilde{w}),(v_1,v_2)) &= \displaystyle\int_{\ic}^T \int_\Omega 
 \tilde{y}_{t}(v_1)_t+\nu \nabla \tilde{w} \nabla v_1  + \frac{1}{\alpha} \tilde{y} v_1  -\nu \nabla \tilde{y} \nabla v_2 + \tilde{w} v_2 dxdt  \\
 & \qquad + \displaystyle\int_\Omega  \nu \nabla \tilde{y}(T)\nabla v_1(T)  dx
 \end{align*}
and linear form
$$ L_M: X \to \mathbb{R}, \quad L_M(v_1,v_2) = \displaystyle\int_{\ic}^T \int_\Omega \tilde{y}_d v_1 dxdt.$$
 
\begin{dfntn}
 The weak formulation of the mixed formulation \eqref{2ord_decoupled} is given by: find $(\tilde{y},\tilde{w}) \in X$ which satisfies
 \begin{equation}\label{decoupled_weak}
  A_M((\tilde{y},\tilde{w}),(v_1,v_2)) = L_M(v_1,v_2) \quad \forall (v_1,v_2) \in X.
 \end{equation}

 \end{dfntn}
 
 By analogy with Theorem~\ref{thm:existence} and Theorem~\ref{thm:unique} it can be shown that the mixed variational form \eqref{decoupled_weak} admits at most one solution and that the pair $(\tilde{y},\tilde{w})$ with $\tilde{y}$ denoting the unique solution to \eqref{2ord_hom} and $\tilde{w}:=-\nu \Delta \tilde{y}$ is a solution to the mixed variational form \eqref{decoupled_weak}. This means that the unique solution to \eqref{2ord_hom} defines the solution to the mixed variational form \eqref{decoupled_weak}.\\

Note that 
 $$ A_M((y,w),(y,w)) =  \displaystyle\int_{\ic}^T \int_\Omega 
 y_t^2   + \frac{1}{\alpha}y^2  + w^2 dxdt + \displaystyle\int_\Omega  \nu |\nabla y(T) |^2  dx.
$$
holds. For this reason, we define an energy norm associated with the bilinear form $A_M$ by
  $$ |||(y,w)|||:= \left( \int_{\ic}^T \int_\Omega y_t^2 + \frac{1}{\alpha}  y^2 + w^2 dxdt \right)^{1/2}.$$

\subsection{A-posteriori error estimate for the semi-time discrete mixed form}\label{sec2:timedisc}

Let us now consider a semi-time discretization of \eqref{decoupled_weak} with respect to $\tilde{y}$ while the variable $\tilde{w}$ is kept continuous. We introduce a time grid $0=\tilde{\tau}_0 < \tilde{\tau}_1 < \dots < \tilde{\tau}_m = T$  with $m \in \mathbb{N}$, time step sizes $\Delta \tilde{\tau}_i = \tilde{\tau}_i - \tilde{\tau}_{i-1}$ and time intervals $I_i = (\tilde{\tau}_{i-1},\tilde{\tau}_i]$ for $i=1,\dots,m$. The time discrete space $V^k$ is defined by
$$ V^{k} = \{v \in C^0((\ic,T);H^1(\Omega)): v|_{I_i} \in \mathbb{P}_1(I_i)\},$$
where $\mathbb{P}_1$ denotes the space of linear polynomials. We set $Y^{k}:=V^k \cap Y$.

\begin{dfntn}(Semi-time discrete mixed form)
The semi-time discrete mixed variational form reads as: find $(\tilde{y}^k,\tilde{w}^k) \in Y^k \times W$ such that
 \begin{equation}\label{weak_dis}
A_M((\tilde{y}^k,\tilde{w}^k),(v_1,v_2)) = L_M(v_1,v_2) \quad \forall (v_1,v_2) \in Y^k\times W.
 \end{equation}
 \end{dfntn}
 
  With arguments similar to those used for \eqref{decoupled_weak} we may show that problem \eqref{weak_dis} admits a unique solution.
 
 Let us now derive a residual based error estimate for the semi-time discrete mixed form \eqref{weak_dis}. We associate with $(\tilde{y}^k,\tilde{w}^k)$ the residuals $R_1^k \in Y^*$ and $R_2^k \in W^*$ by
\begin{equation}\label{r1k}
 R_1^k(v_1) = \displaystyle\int_{\ic}^T \int_\Omega \tilde{y}_d v_1  -  
 (\tilde{y}^k)_t (v_1)_t - \nu \nabla \tilde{w}^k \nabla v_1  -\frac{1}{\alpha} \tilde{y}^k v_1 dxdt
 - \displaystyle\int_\Omega  \nu \nabla \tilde{y}^k(T)\nabla v_1(T)   dx
\end{equation}
and 
\begin{equation}\label{r2k}
 R_2^k(v_2) = \displaystyle\int_{\ic}^T \int_\Omega  \nu \nabla \tilde{y}^k \nabla v_2 - \tilde{w}^k v_2  dxdt.
\end{equation}
Next, we derive $L^2$-representations of $R_1^k$ and $R_2^k$ by elementwise integration by parts
\begin{align*}
  R_1^k(v_1) &= \displaystyle\sum_{i=1}^m \int_{I_i} \int_\Omega \left\{ \tilde{y}_d + (\tilde{y}^k)_{tt} + \nu \Delta \tilde{w}^k   - \frac{1}{\alpha} \tilde{y}^k \right\}  v_1  dxdt\\
  & \qquad +  \sum_{i=1}^m \int_\Omega (\tilde{y}^k)_t v_1 \bigg\vert_{I_i} dx +  \int_\Omega  \nu \Delta \tilde{y}^k(T)   v_1(T)  dx 
 \end{align*}
and
$$R_2^k(v_2) = \sum_{i=1}^m \int_{I_i} \int_\Omega \left\{ -\nu \Delta \tilde{y}^k - \tilde{w}^k \right\} v_2 dxdt. $$
\noindent The residual $R_1^k$ fulfills the Galerkin orthogonality 
\begin{equation}\label{R1k_zero}
R_1^k(v_1) = 0 \quad \forall v_1 \in Y^k
\end{equation}
and it further holds true
\begin{equation}\label{R2k_zero}
 R_2^k(v_2) = 0 \quad \forall v_2 \in W.
\end{equation}

\noindent Moreover, for $(\tilde{y},\tilde{w})\in Y \times W$ and $(\tilde{y}^k,\tilde{w}^k)\in Y^k \times W$ it holds for all $(v_1,v_2) \in Y^k \times W$:
\begin{equation*}
 \begin{array}{r c l}
  A_M((\tilde{y}-\tilde{y}^k,\tilde{w}-\tilde{w}^k),(v_1,v_2)) & = & A_M((\tilde{y},\tilde{w}),(v_1,v_2)) - A_M((\tilde{y}^k,\tilde{w}^k),(v_1,v_2))\\[0.5em]
  & = & L_M(v_1,v_2) - A_M((\tilde{y}^k,\tilde{w}^k),(v_1,v_2)) \quad = 0.\\[0.5em]
 \end{array}
\end{equation*}
Further, the residual equation holds true for all $(v_1,v_2) \in Y \times W$:
\begin{equation}\label{reseq}
 \begin{array}{r c l}
  A_M((\tilde{y}-\tilde{y}^k,\tilde{w}-\tilde{w}^k),(v_1,v_2)) 
  & = & R_1^k(v_1) + R_2^k(v_2)  \quad = R_1^k(v_1),\\[0.5em]
 \end{array}
\end{equation}
where the last step follows from \eqref{R2k_zero}. Now, we are in the position to derive a temporal residual based a-posteriori error estimate for the semi-time discrete mixed variational formulation \eqref{weak_dis}.

\begin{thrm}\label{thm:errest} 
Let $(\tilde{y},\tilde{w}) \in X$ denote the solution to \eqref{decoupled_weak} and let $(\tilde{y}^k,\tilde{w}^k) \in Y^k\times W$ denote the solution to \eqref{weak_dis}. Then, the following residual based a-posteriori error estimate holds true:
\begin{equation}\label{est-thm31}
 ||| (\tilde{y}-\tilde{y}^k,\tilde{w}-\tilde{w}^k)|||^2 \leq C \eta^2, 
\end{equation}
with a constant $C>0$ and
\begin{equation}\label{res:2ord}
 \eta^2 = \sum_{i=1}^m \int_{I_i} \int_\Omega  (\Delta \tilde{\tau}_i)^2 \left| \tilde{y}_d + (\tilde{y}^k)_{tt} + \nu \Delta \tilde{w}^k -  \frac{1}{\alpha} \tilde{y}^k \right|^2 dxdt.
\end{equation}
\end{thrm}

\textit{Proof.}  We combine \eqref{R1k_zero} together with \eqref{reseq}. Let for $v_1 \in Y$ be $I_Y^k v_1$ the approximation to $v_1$ from $Y^k$. Then, it is
\begin{equation*}
 \begin{array}{r c l}
  A_M((\tilde{y}-\tilde{y}^k,\tilde{w}-\tilde{w}^k),(v_1,v_2))  =  R_1^k(v_1 - I_Y^k v_1) =\\
   \displaystyle\sum_{i=1}^m \int_{I_i} \int_\Omega r_{1,int}^k(v_1 - I_Y^k v_1) dxdt +  \int_\Omega \nu \Delta \tilde{y}^k(T) (v_1 - I_Y^k v_1)(T) dx \\
   + \displaystyle\sum_{i=1}^m \int_\Omega (\tilde{y}^k)_t (v_1-I_Y^k v_1)\bigg\vert_{I_i} dx,
 \end{array}
\end{equation*}
where we use the notation $r_{1,int}^k:= \tilde{y}_d + (\tilde{y}^k)_{tt} + \nu \Delta \tilde{w}^k- \frac{1}{\alpha} \tilde{y}^k$. Note that the last summands vanish since $(v_1-I_Y^k v_1)(\tilde{\tau}_i) = 0$ for $i=0,\dots,m$. We can estimate using Cauchy-Schwarz
$$ |A_M((\tilde{y}-\tilde{y}^k,\tilde{w}-\tilde{w}^k),(v_1,v_2))| \leq \int_\Omega \left( \sum_{i=1}^m \| r_{1,int}^k\|_{L^2(I_i)} \| v_1 - I_Y^k v_1 \|_{L^2(I_i)} \right) dx.$$
Next, using standard interpolation properties (see e.g.\ \cite[Theorem 1.7]{AinO00}), we arrive at
\begin{equation*}
 |A_M((\tilde{y}-\tilde{y}^k,\tilde{w}-\tilde{w}^k),(v_1,v_2))| \leq \int_\Omega \left( \sum_{i=1}^m \| r_{1,int}^k \|_{L^2(I_i)} \; c_1 \; \Delta \tilde{\tau}_i | v_1 |_{H^1(\tilde{I}_i)} \right) dx,
\end{equation*}
where $\tilde{I}_i$ denotes the set of intervals which share a vertex with $I_i$. We recall that $| . |_{H^1}$ denotes the $H^1$-seminorm. Together with the Cauchy-Schwarz inequality for sums, we arrive at
\begin{equation}\label{est:AM}
\begin{array}{l l l}
  |A_M((\tilde{y}-\tilde{y}^k,\tilde{w}-\tilde{w}^k),(v_1,v_2))| \\
 \qquad\qquad\qquad\qquad \leq  c_1 \displaystyle\int_\Omega \left( \sum_{i=1}^m \| r_{1,int}^k \|_{L^2(I_i)}^2 (\Delta \tilde{\tau}_i)^2 \right)^{1/2} \left( \sum_{i=1}^m | v_1 |_{H^1(\tilde{I}_i)}^2 \right)^{1/2} dx   \\
\qquad\qquad\qquad\qquad   \leq  c_2 \displaystyle\int_\Omega \left( \sum_{i=1}^m \| r_{1,int}^k \|_{L^2(I_i)}^2 (\Delta \tilde{\tau}_i)^2 \right)^{1/2}  | v_1 |_{H^1(0,T)} dx \\
\qquad\qquad\qquad\qquad   \leq  c_2  \left(\displaystyle\int_\Omega \sum_{i=1}^m \| r_{1,int}^k \|_{L^2(I_i)}^2 (\Delta \tilde{\tau}_i)^2 dx \right)^{1/2} \left( \displaystyle\int_\Omega | v_1 |_{H^1(0,T)}^2 dx \right)^{1/2},
\end{array}
\end{equation}
where we use H\"older's inequality in the last step. We note that  
$$\left( \displaystyle\int_\Omega | v_1 |_{H^1(0,T)}^2 dx \right)^{1/2} \leq \left( \int_0^T \int_\Omega (v_1)_t^2 + \frac{1}{\alpha} v_1^2 + v_2^2 dxdt \right)^{1/2} = ||| (v_1,v_2)|||.$$ 
In \eqref{est:AM} we choose $v_1:= \tilde{y}-\tilde{y}^k$ and $v_2 := \tilde{w}-\tilde{w}^k$ and denote $e:=(\tilde{y}-\tilde{y}^k,\tilde{w}-\tilde{w}^k)$, which leads to 
$$ |A_M(e,e)| \leq c_2 \left( \int_\Omega \sum_{i=1}^m \|r_{1,int}^k \|_{L^2(I_i)}^2 (\Delta \tilde{\tau}_i)^2 dx \right)^{1/2} \cdot ||| e |||.$$
By the definition of the energy norm $||| \cdot |||$, it follows $A_M(e,e) \geq |||e|||^2$ which yields the a-posteriori error estimate
$$ |||e |||^2 \leq  C \left( \int_\Omega \sum_{i=1}^m \|r_{1,int}^k \|_{L^2(I_i)}^2 (\Delta \tilde{\tau}_i)^2 dx \right). \hspace{1cm}\square$$

\begin{rmk}[Adaptive cycle]
 In order to construct an adaptive time grid, we follow the standard 
 \begin{center}
solve $\to$ estimate $\to$ mark $\to$ refine  
 \end{center}
 cycle. In practice, we solve \eqref{weak_dis} using rectangular space-time finite elements. Then, the error in each time interval is estimated using \eqref{est-thm31}. The intervals with the largest errors are marked using the D\"orfler marking strategy \cite{Doe96}. For refinement, we perform a bisection of the marked intervals. We iterate this loop until the time grid has a prescribed number of e.g.\ $N$ time instances.
\end{rmk}
\begin{rmk}[Heuristic assumption]\label{rem:heuristic}
Note that we derived an error estimate \eqref{est-thm31} for a time discrete formulation in $y$ whereas $w$ is kept continuous. In practice, we solve a fully space-time discrete mixed variational formulation, but still use the error estimate for the semi-time discrete form to construct an adaptive time grid. For this, we assume that the temporal discretization of $y^k$ is insensitive with respect to the spatial discretization. In fact, numerical studies in \cite{AllGH16,AllGH18} show that temporal and spatial discretization decouple for the considered problem settings. In addition, we also assume that a temporal discretization of $w^k$ does not strongly influence the error estimate. Of course, these heuristic assumptions might not hold in general. For this reason, we will in future research derive a-posteriori error estimates for a fully space-time discrete mixed variational form.
\end{rmk}

With the help of \eqref{est-thm31}, we are able to refine the time grid by means of the residual of the system \eqref{2ord_decoupled}. This property will constitute the major building block for the time-adaptive approach in the MPC framework as discussed in the next Sections~\ref{secadaptivity}~and~\ref{secnumerics}.

\subsection{State equation with depletion term}\label{secdepletion}
Let us now consider an optimal control problem of the form \eqref{mpcocpreduced}, where an additional depletion term in the state equation appears as
\begin{equation}\label{heatdepletion}
\left\{
\begin{array}{rcll}
y_t-\nu\Delta y - \mu y& = & f+u &\text{ in } (\ic,T]\times\Omega,\\
y &= & 0 &\text{ on } [\ic,T] \times \partial \Omega,\\
y(\ic) & = & y_0 & \text{ in } \Omega,
\end{array}
\right.
\end{equation}
with $\mu > 0$. 
The reformulation of the associated optimality system into an elliptic equation and an associated mixed formulation, respectively, follows along the lines of Sections~\ref{sec2:subreform} and \ref{secmixedform}. In particular, the mixed formulation reads as
\begin{equation}\label{2ord_decoupled_depletion}
 \left\{
 \begin{array}{rcll}
 -\tilde{y}_{tt}-\nu \Delta \tilde{w} + 2 \nu \mu \Delta \tilde{y} +\left(\frac{1}{\alpha}+\mu^2\right)\tilde{y} & = &\tilde{y}_d & \text{ in } (\ic,T)\times\Omega,\\
 \nu \Delta \tilde{y} + \tilde{w} & = & 0 &  \text{ in } (\ic,T)\times\Omega,\\
 \tilde{y} & = & 0 &\text{ on } [\ic,T]\times\partial\Omega,\\
  \tilde{w} & = & 0 &\text{ on } [\ic,T]\times\partial\Omega,\\
 \left(\tilde{y}_t-\nu \Delta \tilde{y} - \mu \tilde{y}\right)(T) & = & 0 &\text{ in }\Omega,\\
 \tilde{y}(\ic) & = & 0 &\text{ in } \Omega.
 \end{array}
 \right.
 \end{equation}
Let us define the bilinear form
 $$ A_M^\mu:X \times X \to \mathbb{R},$$
$$A_M^\mu((\tilde{y},\tilde{w}),(v_1,v_2)) = \displaystyle\int_{\ic}^T \int_\Omega 
( \tilde{y}_{t}(v_1)_t+\nu \nabla \tilde{w} \nabla v_1 - 2 \nu \mu \nabla \tilde{y} \nabla v_1 +\left(\frac{1}{\alpha}+\mu^2\right)\tilde{y} v_1  $$  
 $$ -\nu \nabla \tilde{y} \nabla v_2 + \tilde{w} v_2 )dxdt  +  \displaystyle\int_\Omega  \nu \nabla \tilde{y}(T)\nabla v_1(T) - \mu \tilde{y}(T) v_1 (T) dx
$$
and linear form 
$$ L_M^\mu: X \to \mathbb{R}, \quad L_M^\mu(v_1,v_2) = \displaystyle\int_{\ic}^T \int_\Omega \tilde{y}_d v_1 dxdt,$$
where $\tilde{y}_d:= \frac{1}{\alpha} y_d - f_t - \nu \Delta f - \mu f + g_{tt} - \nu^2 \Delta^2 g - 2 \nu \mu \Delta g - (\frac{1}{\alpha} + \mu^2) g$. 
\begin{dfntn}
 The weak formulation of the mixed formulation \eqref{2ord_decoupled_depletion} is given by: find $(\tilde{y},\tilde{w}) \in X$ which satisfies
 \begin{equation}\label{decoupled_weakdepletion}
  A_M^\mu((\tilde{y},\tilde{w}),(v_1,v_2)) = L_M^\mu(v_1,v_2) \quad \forall (v_1,v_2) \in X.
 \end{equation}
 \end{dfntn}

The semi-time discrete mixed variational formulation then reads as
\begin{equation}\label{semitimediscrete_depletion}
 A_M^\mu((\tilde{y}^k,\tilde{w}^k),(v_1,v_2)) = L_M^\mu(v_1,v_2) \quad \forall (v_1,v_2) \in Y^k \times W.
\end{equation}

With similar arguments as in the previous sections, one can show existence of a unique solution of the involved equations provided sufficient regularity of the data.
 
In analogy to Theorem~\ref{thm:errest} we can derive a temporal residual based a-posteriori error estimate for \eqref{semitimediscrete_depletion}.

\begin{thrm}\label{thm:errest_depl} 
Let $(\tilde{y},\tilde{w}) \in X$ denote the solution to \eqref{decoupled_weakdepletion} and let $(\tilde{y}^k,\tilde{w}^k) \in Y^k\times W$ denote the solution to \eqref{semitimediscrete_depletion}. Further, let $\mu \leq \nu/c_p^2$, where $c_p$ denotes the Poincar\'{e} constant. Then, the following residual based a-posteriori error estimate holds true:
\begin{equation}\label{est-thm31mu}
 ||| (\tilde{y}-\tilde{y}^k,\tilde{w}-\tilde{w}^k)|||^2 \leq C \eta^2, 
\end{equation}
with a constant $C>0$ and
\begin{equation}\label{res:2ordmu}
 \eta^2 = \sum_{i=1}^m \int_{I_i} \int_\Omega (\Delta \tilde{\tau}_i)^2 \left| \tilde{y}_d + (\tilde{y}^k)_{tt} + \nu \Delta \tilde{w}^k -2 \nu \mu \Delta \tilde{y}^k -  \left(\frac{1}{\alpha} +\mu^2 \right)\tilde{y}^k \right|^2 dxdt.
\end{equation}
\end{thrm}
\begin{proof}
The proof follows along the lines of the proof of Theorem~\ref{thm:errest}. Note that it holds
\begin{align}\label{eqthm}
 A_M^\mu((\tilde{y},\tilde{w}),(\tilde{y},\tilde{w})) &= \int_0^T \int_\Omega \tilde{y}_t^2 -2 \nu \mu |\nabla \tilde{y}|^2 + \left(\frac{1}{\alpha} + \mu^2 \right) \tilde{y}^2 + \tilde{w}^2 dxdt \nonumber\\
&+  \int_\Omega \nu |\nabla \tilde{y}(T)|^2 - \mu |\tilde{y}(T)|^2 dx.
\end{align}
Using Green's formula, the definition of $\tilde{w}$ and Young's inequality, we can estimate the second summand in \eqref{eqthm} by
\begin{equation*}
 \begin{array}{r c l}
  \displaystyle\int_0^T \int_\Omega - 2 \nu \mu |\nabla \tilde{y}|^2 dxdt  =  \displaystyle\int_0^T \int_\Omega 2 \nu \mu \Delta \tilde{y} \; \tilde{y} \;dxdt = \int_0^T \int_\Omega -2  \mu \tilde{w} \tilde{y} \; dxdt \\
  \geq \displaystyle\int_0^T \int_\Omega -2  \mu |\tilde{w}| \; |\tilde{y}| \; dxdt
   \geq  \displaystyle\int_0^T \int_\Omega - 4 \delta \mu^2 \tilde{y}^2 - \frac{1}{4\delta} \tilde{w}^2 \; dxdt.
 \end{array}
\end{equation*}
With the choice $\delta:= \displaystyle\frac{1+2\alpha\mu^2}{8 \alpha \mu^2}$, it holds that $ -4\delta \mu^2 + \frac{1}{\alpha} + \mu^2 \geq 0$ and $ -\frac{1}{4\delta} + 1 \geq 0$.\\

Using the Poincar\'{e} inequality, we can estimate the last term in \eqref{eqthm} by
$$\int_\Omega \nu |\nabla y(T)|^2 - \mu |y(T)|^2 dx \geq \left( \frac{\nu}{c_p^2} - \mu \right) \|  y(T) \|_{L^2(\Omega)}^2$$
with Poincar\'{e} constant $c_p$. If $\mu \leq \nu / c_p^2$, then $\displaystyle\int_\Omega \nu |\nabla y(T)|^2 - \mu |y(T)|^2 dx \geq 0$. Thus, for $\mu \leq \nu / c_p^2$ it holds that
$$ A_M^\mu((\tilde{y},\tilde{w}),(\tilde{y},\tilde{w})) \geq ||| (\tilde{y},\tilde{w})|||^2.$$
With this, the a-posteriori error estimate follows in analogy to Theorem~\ref{thm:errest}.
\end{proof}

\subsection{Control constraints, abstract controls and state constraints}\label{constraints}
The case of partially supported controls and control constraints can be treated by switching to an elliptic system for the adjoint state $p$. In particular, we can consider linear and bounded control operators $B: U \to L^2((0,T);H^{-1}(\Omega))$ mapping controls to feasible right hand sides, where $U$ denotes a real Hilbert space, and control constraints $u \in U_{\text{ad}} \subseteq U$ with $U_{\text{ad}}$ describes a convex, bounded and closed set of admissible controls. Under the corresponding regularity assumptions similar to those in  Lemma~\ref{lemma:regularity}, the associated optimality system can be reformulated into an elliptic equation of the form
\begin{equation}\label{elliptic_p}
 \left\{
 \begin{array}{rcll}
 -p_{tt} +\nu^2 \Delta^2 p - B\mathbb{P}_{U_{\text{ad}}}\left\{-\frac{1}{\alpha} B^*p\right\} & = & f + \nu \Delta y_d - (y_d)_t & \text{ in } (\ic,T)\times\Omega,\\
 p & = & 0 &\text{ on } [\ic,T]\times\partial\Omega,\\
 \nu \Delta p  & = & y_d &\text{ on } [\ic,T]\times\partial\Omega,\\
 (-p_t - \nu \Delta p)(0) & = & y_0 - y_d(0) &\text{ in }\Omega,\\
 p(T) & = & 0 &\text{ in } \Omega.
 \end{array}
 \right.
 \end{equation}
 with $B^*$ denoting the dual operator to $B$ and $\mathbb{P}_{U_{\text{ad}}}$ denoting the projection operator onto the admissible control space. An a-posteriori error estimate can be derived analogously, see \cite{AllGH18} for more details. Using a regularization of the projection operator, it is also possible to derive an elliptic equation for the state, see \cite{NeiPS11}.\\
 Further, we note that the procedure above can be extended to the treatment of state constraints by e.g.\ adapting the approach of \cite{LiuGY09}. This is to consider the reduction to the elliptic space-time formulation for the state obeying state constraints. However, for the proof of concept we, in the present work, avoid the incorporation of additional constraints and other practical relevant control operators.

 \section{Time adaptivity in MPC}\label{secadaptivity}
In this section, we propose the use of a time-adaptive technique within MPC. In the classical application of MPC algorithms the length of the application horizon is fixed a priori to the length $\tau:= \frac{\bar T}{N-1}$, and the prediction horizon $[t_i,t_i^N]$ is discretized equidistantly with a time grid containing $N$ time instances $t_i^j:=t_i+(j-1)\tau$ for $j=1,\dots,N$. This might not be ideal in practice. 
Therefore, we here would like to reply to the following questions:
\medskip

 \textit{(i) How to choose a time discretization for the prediction horizon $[t_i,t_i+\bar T]$ in each level $i$ of the MPC?}
 
 \textit{(ii) How to choose the length of the application horizon $\tau_i$ in each level $i$ of the MPC to implement the feedback control?}
\medskip

We aim at computing the temporal discretization to identify the important dynamical structures according to the optimization goal. We propose an adaptive strategy which avoids unnecessary small uniform temporal discretizations and realize an efficient implementation. The proposed approach will lead to adaptive time discretizations which are related to the \textit{optimal} state for each of the MPC subproblems.Thus, we directly obtain the application horizon where we compute the feedback map.

The idea of adaptivity leads to different combinations using the error estimate \eqref{est-thm31mu}. Here, we will deal with an adaptive grid in each subinterval for a fixed prediction horizon where the time discretization is computed on the fly. For a different adaptive concept based on goal-oriented adaptivity, see the recent work \cite{GSSnew}.

Therefore, for a given prediction interval $[t_i,t_i^{N}]$ at each MPC iteration $i$, we make use of the a-posteriori error estimation {\color{black} \eqref{est-thm31mu}} for the state to compute an adaptive time grid within the current time horizon. Note that $t_0 := 0$ is the initial time. The scheme is visualized in Figure~\ref{fig:MPConline}.
 
  \begin{figure}[htbp]
   \centering
    \begin{tikzpicture}[scale=1.8]
       \draw[->,gray,line width=0.5mm] (-1.3,0) -- (4.7,0);
       
         \fill[blue] (-0.8,-0.1) rectangle (-0.85,0.1);    
          \fill[blue] (0.1,-0.1) rectangle (0.15,0);
          \fill[purple] (0.1,0) rectangle (0.15,0.1);
          \fill[purple] (0.46,-0.1) rectangle (0.51,0.1);
          \fill[blue] (0.52,-0.1) rectangle (0.57,0.1);
          \fill[purple] (0.8,-0.1) rectangle (0.85,0.1);
          \fill[blue] (0.87,-0.1) rectangle (0.92,0.1);
          \fill[blue] (1.6,-0.1) rectangle (1.65,0.1);
         \fill[purple] (1.8,-0.1) rectangle (1.85,0.1);
          \fill[purple] (2.8,-0.1) rectangle (2.85,0.1);
          \fill[blue] (3,-0.1) rectangle (3.05,0.1);
          \fill[purple] (3.92, -0.1) rectangle (3.97,0.1);
          \draw[line width=0.5mm] (-0.9,-0.25) node {\color{blue} $t_i$};
          \draw[line width=0.5mm] (3,-0.25) node {\color{blue} $t^N_{i}$};
          \draw[blue,line width=0.5mm] (-0.8,-0.5) -- (3.02,-0.5);
          \draw[blue,line width=0.5mm] (-0.8,-0.5) -- (-0.8,-0.35);
          \draw[blue,line width=0.5mm] (3.02,-0.5) -- (3.02,-0.35);
          \draw (1,-0.7) node {\footnotesize \color{blue} prediction horizon};
          \draw[purple,line width=0.5mm] (0.1,0.55) -- (3.92,0.55);
          \draw[purple,line width=0.5mm] (0.1,0.4) -- (0.1,0.55);
          \draw[purple,line width=0.5mm] (3.92,0.4) -- (3.92,0.55);
          \draw (2,0.7) node {\footnotesize \color{purple} shifted prediction horizon };
          \draw (4, 0.25) node{\color{purple} $t_{i+1}^{N}$};
          \draw (0.15, 0.25) node{\color{purple} $t_{i+1} = t_i+\tau_i$};
          \draw (5,0) node {\footnotesize \color{gray} time};
       \end{tikzpicture}
       \caption{Scheme of adaptive approach: The blue color refers to the grid at iteration $i$ starting at time $t_i$ till $t_i^N$ whereas the red color refers to the next MPC level $i+1$. We note that the only guaranteed overlap of the time grids is for second time instance at iteration $i$ which corresponds to the first time instance at iteration $i+1$.}\label{fig:MPConline}
       \end{figure}
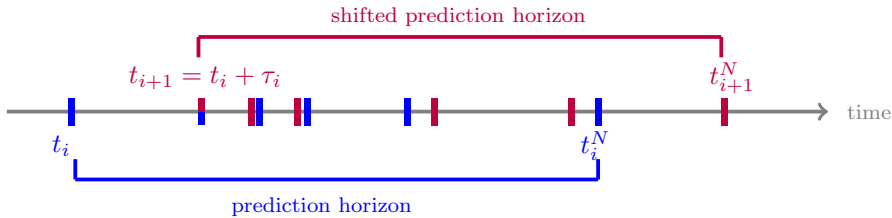
 For a given number of degrees of freedom $N$ the algorithm distributes the time instances within the prediction horizon $[t_i,t_i^N]$ according to the error estimation \eqref{est-thm31mu}, where we assume that all prediction horizons have the same length $t_i^{N}-t_i = \bar{T}$. The resulting adaptive time grid at each time instance $t_i$ is related to the \textit{optimal} state of the corresponding open loop subproblem of the current MPC step. Again, we assume that the heuristic assumptions of Remark~\ref{rem:heuristic} hold true which enables an efficient computation. The approach is summarized  in Algorithm~\ref{Alg:NMPC_on} in Section \ref{sec:intro} .

 \begin{rmk}[Warm start]
  In order to make computations even more efficient, the information of the previous MPC iteration can be used as a \textit{warm} start for the next MPC iteration. In particular, after a coarsening step of the previous adaptive time grid, this grid can be used as an initial adaptive time grid for the next prediction horizon. Furthermore, to improve the inner open-loop solver in each iteration {\color{black} one can use} as initial control the one computed at the previous step.
 \end{rmk}

\begin{rmk}[Efficiency under perturbations]
This approach allows to compute the best time grid for every iteration of the MPC method. The grid will, in general, not result to be equidistant. This approach is particular sensitive to perturbations on the system. Therefore, it will automatically react and deliver  an adaptive grid taking into account the current measurements that might or might not contain an error. 
\end{rmk}

\section{Numerical example}\label{secnumerics} In the following tests, we investigate numerically the  time-adaptive MPC algorithm proposed in Section~\ref{secadaptivity}. In all numerical examples, the considered spatial domain is the open interval $\Omega = (0,1)$. In order to solve the mixed form \eqref{2ord_decoupled_depletion}, we introduce a partitioning of the space-time domain into regular orthotopes and use $\mathbb{Q}_1$ space-time finite elements for discretization, where $\mathbb{Q}_1$ is the space of polynomials of separate degree up to $1$. We solve the equation with a direct solver using a coarse spatial resolution. For the solution of the MPC open loop subproblems, we use an implicit Euler scheme for the temporal discretization and use piecewise linear and continuous finite elements for the spatial discretization. The optimal control problem is solved with a direct solver, where we take a fine spatial resolution an equidistant discretization with $\Delta x = 1/100$. All coding is done in \textsc{Matlab R2019}a.

\subsection{Test 1: State equation with depletion term and random disturbances}

In this numerical test, we consider an optimal control problem where the state dynamics are governed by \eqref{heatdepletion} with $\mu > 0$. Let us note that the Poincar\'{e} constant $c_p$ and the first eigenvalue $\lambda_1$ of the Laplace-Dirichlet operator are related by $\lambda_1 = 1/c_p^2$ (see, e.g., \cite[Proposition 8.4.3]{AttBM14}). For the considered domain $\Omega = (0,1)$, the first eigenvalue $\lambda_1$ of the Laplace-Dirichlet operator is given by $\lambda_1 = \pi^2$ (see, e.g., \cite[Proposition 8.5.2]{AttBM14}). Then, since Theorem~\ref{thm:errest_depl} is applicable if $\mu \leq \nu / c_p^2$, for this setting it requires $\mu \leq \nu \cdot \pi^2$. In this example, we set $\nu =0.1$ and $\mu = 5$. Thus, we consider an unstable case which goes beyond the assumptions of Theorem~\ref{thm:errest_depl}. Nevertheless, we will see that the numerical tests under this configuration still provide satisfactory results, very similar to a stable case with $\mu\leq \nu\pi^2$ as required in Theorem~\ref{thm:errest_depl}. 
The initial condition for the state is chosen as $y_\circ(x) \equiv 0$ and the source term in the state equation is set to $f(t,x)\equiv 0$. The regularization parameter in the cost is chosen as $\alpha = 10^{-3}$ and the desired state is given by 
$$y_d(t,x) = -10 | x- 0.25| - 10 |x-0.75| + 10,$$
which is a stationary state and shown in Figure~\ref{fig:test3yd} (left). Thus, the goal of the optimal control problem is to steer the state $y$, which fulfills \eqref{heatdepletion} in a weak sense, as close as possible to the desired state $y_d$ and keep it there (for an infinite amount of time). In Figure~\ref{fig:test3yd} (middle) we show the controlled state solution using Algorithm~\ref{Alg:NMPC_on} with the choices $N=20, \bar{T}=0.5$ and plot the adaptive time grid for the first prediction horizon $[0,0.5]$ in Figure~\ref{fig:test3yd} (right). 
\begin{figure}[htbp]
\centering
  \includegraphics[scale=0.3]{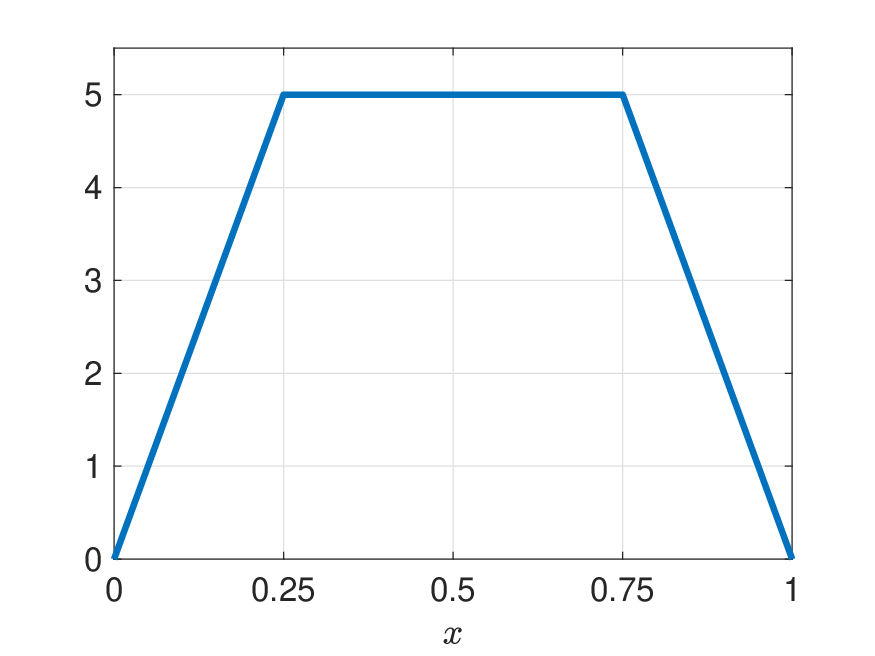} \hspace{-0.4cm}
    \includegraphics[scale=0.3]{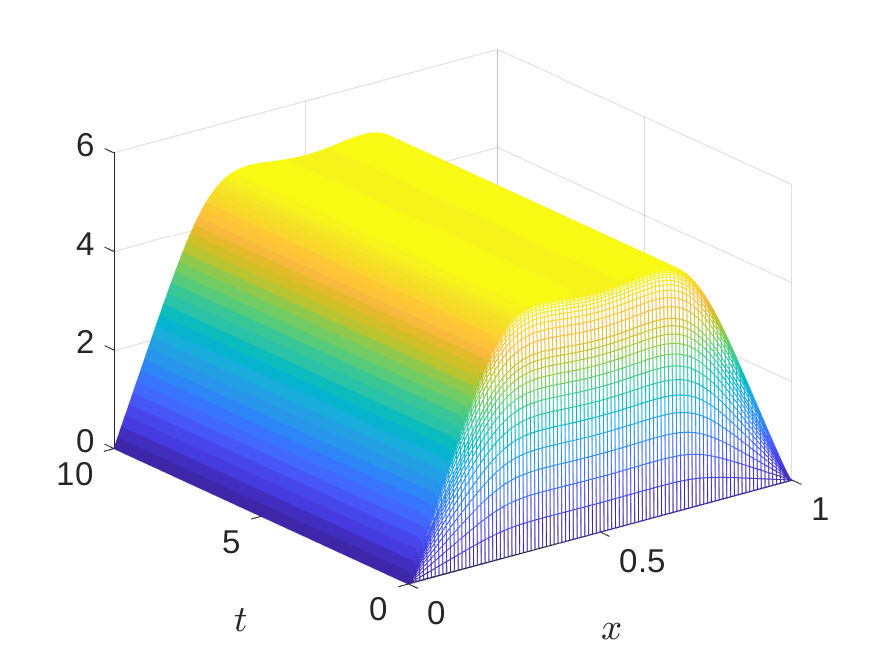} \hspace{-0.4cm}
  \includegraphics[scale=0.3]{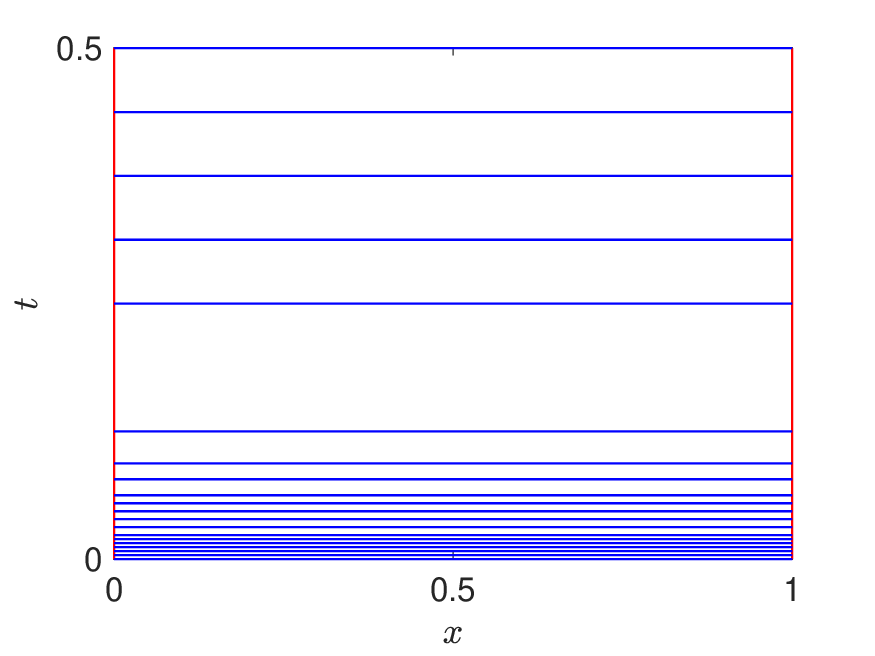}  
\caption{Test 1: Desired state $y_d$ (left), controlled state (middle), adaptive time grid for prediction horizon $[0,0.5]$ (right).}\label{fig:test3yd}
\end{figure}
For a cheap computation of the adaptive time grid, we solve \eqref{2ord_decoupled_depletion} with a coarse spatial resolution of $\Delta x = 1/4$, compare Remark~\ref{rem:heuristic}. We observe a fine temporal discretization toward $t=0$, where the initial state must be steered from $y_\circ(x)=0$ as close as possible to the desired state.\\

\noindent In realistic scenarios, however, often disturbances enter the system, see Figure~\ref{fig:closedloop} for a schematic presentation. 
\begin{figure}[h]
   \centering
    \begin{tikzpicture}[scale=1.5]     
   \node (OCP) at (0,0) [rectangle,fill=gray!50,draw, text width = 4cm,align=center] {\small \color{black}Finite horizon open-loop optimal control problem\\[2ex] $\min \hat{J}^N(u,t_0^i,y_0^i)$};
   \node (Sim) at (0,-1.5) [rectangle,fill=gray!50,draw, text width = 4.8cm,align=center] {\small \color{black}State equation\\[2ex] $y^N(t) = y_{[\phi^N,t_0^i,y_0^i]}(t)$ on $(t_0^i,t_1^i]$};
  \draw[thick,-] (1.5,0) -- (2.2,0);
  \draw[thick,-] (2.2,0) -- node[anchor=west,text width = 3cm]{\small Model predictive feedback value $\phi^N$} (2.2,-1.5);
  \draw[thick,->] (2.2,-1.5) -- (1.8,-1.5);
  \draw[thick,-] (-1.8,-1.5) -- (-2.2,-1.5);
  \draw[thick,-] (-2.2,-1.5) -- node[anchor=east,text width = 2.5cm,align=right]{\small Initial value $y_{i+1}=y^N(t_{i+1})$} (-2.2,0);
    \draw[thick,->] (-2.2,0) -- (-1.5,0);
    \draw[thick,->] (-2.6,0.2) -- (-1.5,0.2);
    \node (in) at (-3.2,0.2) [rectangle,fill=gray!50,draw, text width = 1.3cm,align=center] {\small \color{black}Problem\\ data};
    \draw[thick,->] (0,-2.5) -- node[anchor=north,text width = 1.9cm]{\small \\[1ex] Disturbances} (0,-2);
       \end{tikzpicture}
       \caption{Scheme of MPC with disturbances.}
       \label{fig:closedloop}
       \end{figure}
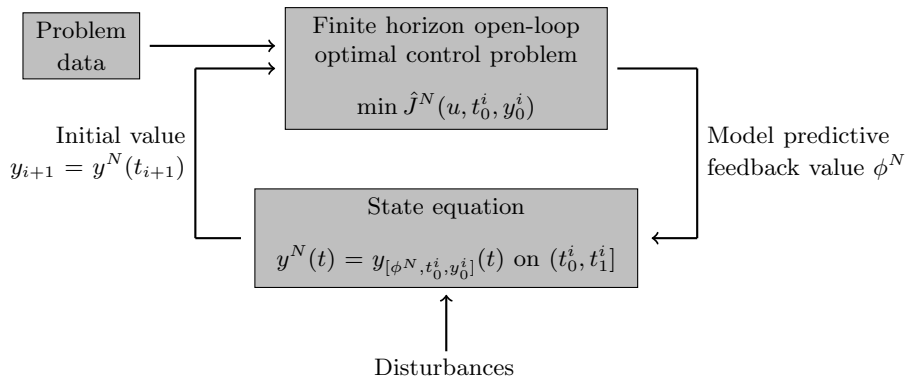  
In particular, we focus on disturbances which happen at random time points $\{\omega_\kappa\}_{\kappa=1}^K$ in the source term $f$ and current state $y_i$ of random magnitudes $\{(\chi_\kappa,\psi_\kappa)\}_{\kappa=1}^K$ leading to a disturbed initial value $y_{i+1}=y^N(t_{i+1})$ for the next MPC loop. In particular, if $\omega_\kappa \in (t_i,t_i + \tau_i]$, i.e.\ if the current simulation window contains one of the random time instances, we consider the following disturbed state equation for implementing the model predictive feedback value:
\begin{equation}\label{heatdepletiondisturbed}
\left\{
\begin{array}{rcll}
y_t-\nu\Delta y - \mu y& = & f_{dist}+\phi^N &\text{ in } (t_i,t_i + \tau_i]\times\Omega,\\
y &= & 0 &\text{ on } (t_i,t_i + \tau_i] \times \partial \Omega,\\
y(t_i) & = & y_i + y_{dist} & \text{ in } \Omega,
\end{array}
\right.
\end{equation}
where $f_{dist}(t,x) \equiv - \chi_\kappa$ in $(t_i,t_i + \tau_i] \times \Omega$ and $y_{dist} (x) = - \psi_\kappa \sin(\pi x)$ in $\Omega$. In this example, we generate the random numbers once and run all tests for these values in order to make the experiments comparable. We consider $K=4$ random time points $\omega_1 = 3.51, \omega_2 = 4.73, \omega_3 = 5.85, \omega_4 = 8.30$ and values $\chi_1 = 75.85, \chi_2 = 380.44, \chi_3 = 567.82, \chi_4 = 753.72$ and $\psi_1 = 6.78, \psi_2 = 7.57, \psi_3 = 7.43, \psi_4 = 3.92$. In Figure~\ref{fig:costmesh} (left) we show the decay of the cost functional for an increasing number of time instances $N$ per prediction horizon for three examples of prediction horizon lengths ($\bar{T}=0.2, 0.3, 0.4$) comparing the adaptive approach of Algorithm~\ref{Alg:NMPC_on} with the standard uniform approach. In this example, we exemplarily run the MPC loop until $t_i=10$ for some $i \in \mathbb{N}$, i.e.\ we cover a time domain of $[0,10]$. We observe in this setting that the adaptive approach delivers smaller cost function values than the equidistant approach. The greatest benefit of the adaptive approach is achieved when a small number of degrees of freedom in a comparatively large prediction horizon is considered, where the adaptive approach distributes the time discretization points tailored to the optimal state dynamics indicated through the error estimate \eqref{est-thm31mu}. Note that fixing $N$ and $\bar{T}$ can lead to different lengths of the simulation window $(t_i,t_i+\tau_i]$ in which the feedback value is applied in Algorithm~\ref{Alg:NMPC_on}, and thus different number of degrees of freedom for the whole considered time domain $[0,10]$. 
\begin{figure}[h]
\centering
   \includegraphics[scale=0.3]{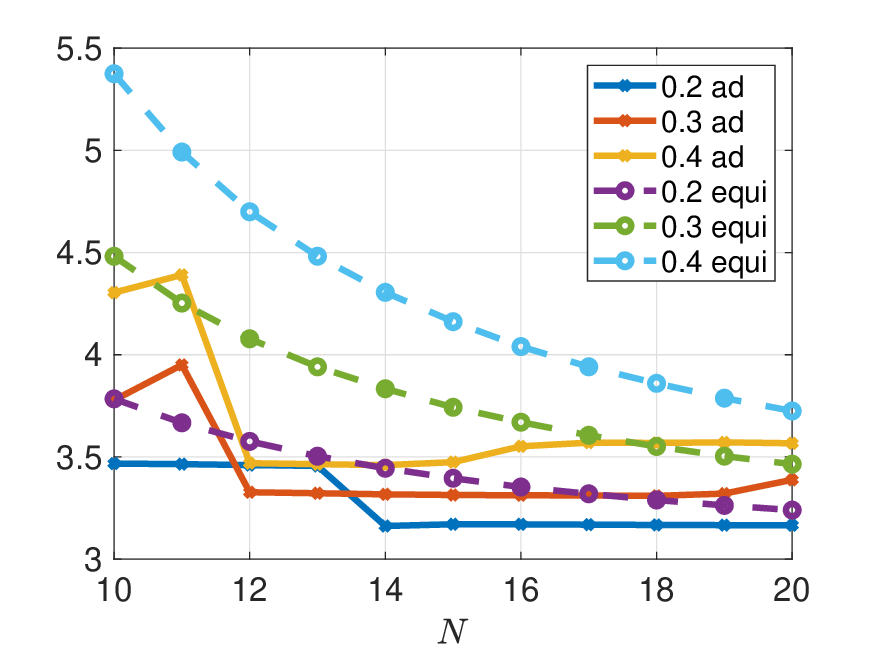}
     \includegraphics[scale=0.3]{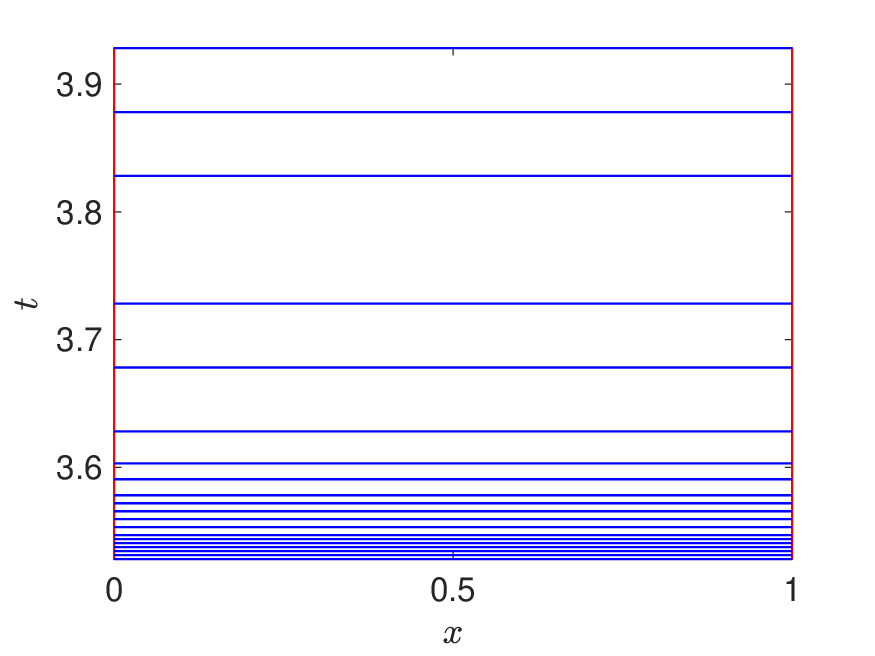}
\caption{Test 1: Decay of cost functional for increasing $N$ for different prediction horizons $\bar{T}$ comparing the adaptive and equidistant approach (left); example of an adaptive prediction horizon for $N=20$, $\bar{T}=0.4$ (right).}\label{fig:costmesh}
\end{figure}
An example of an adaptive prediction horizon is shown in Figure~\ref{fig:costmesh} (right), where we see smaller time steps toward the disturbance at $\omega_1 = 3.51$. Finally, Figure~\ref{fig:control} shows the tracking term and the control costs over time comparing the adaptive with the equidistant method. We observe that the time-adaptive MPC approach leads to a closer tracking of the desired state than an equidistant approach (Figure~\ref{fig:control}, top), but with the price of higher control costs (Figure~\ref{fig:control}, bottom). 
\begin{figure}[h]
\centering
       \includegraphics[scale=0.3]{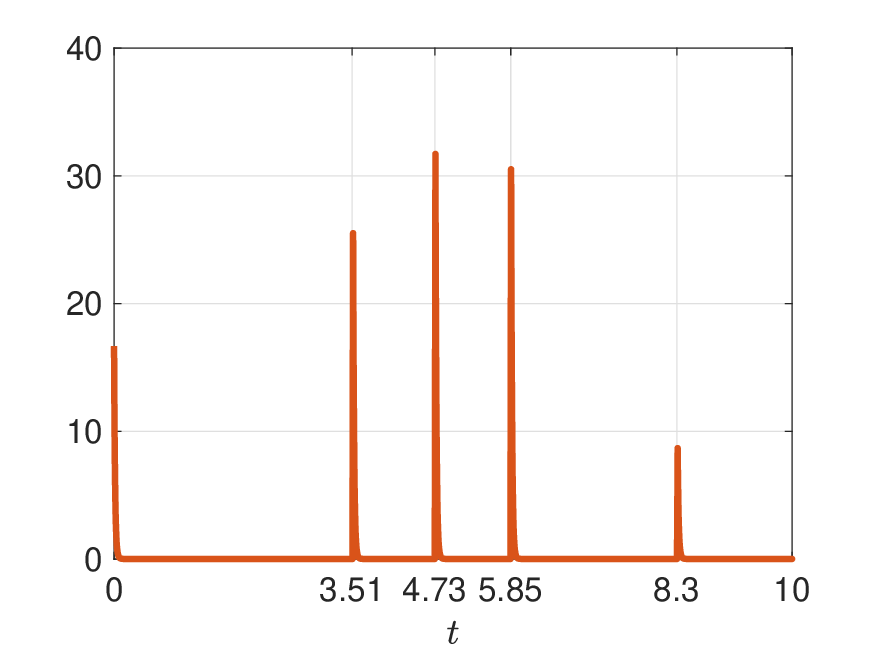} 
     \includegraphics[scale=0.3]{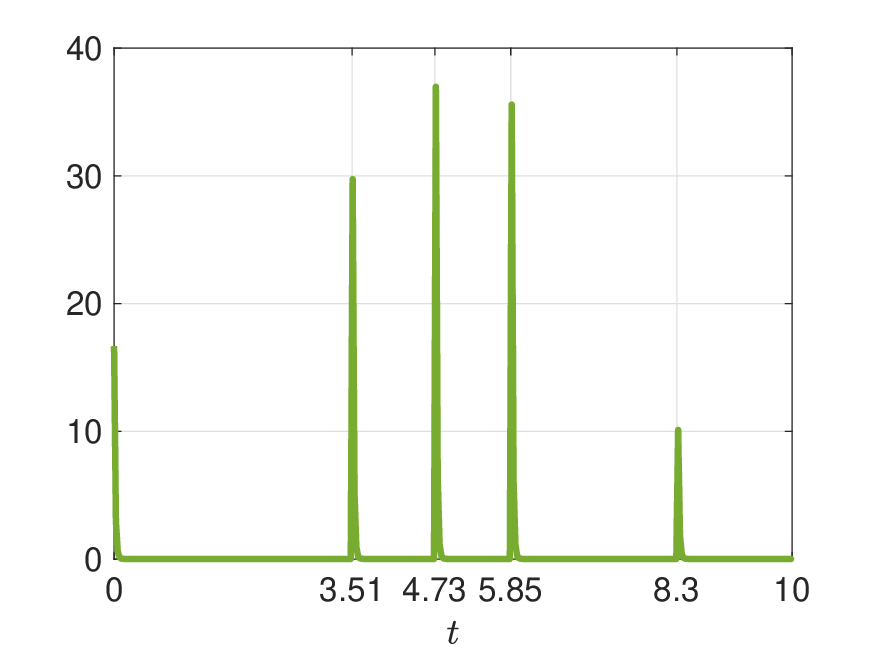} \\
  \hspace{0.1cm} \includegraphics[scale=0.3]{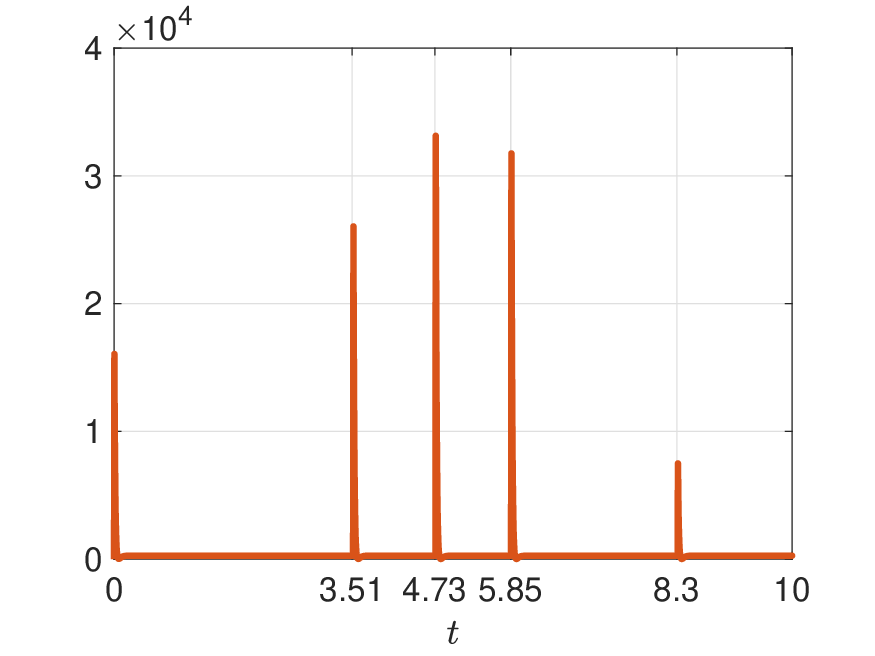} 
     \includegraphics[scale=0.3]{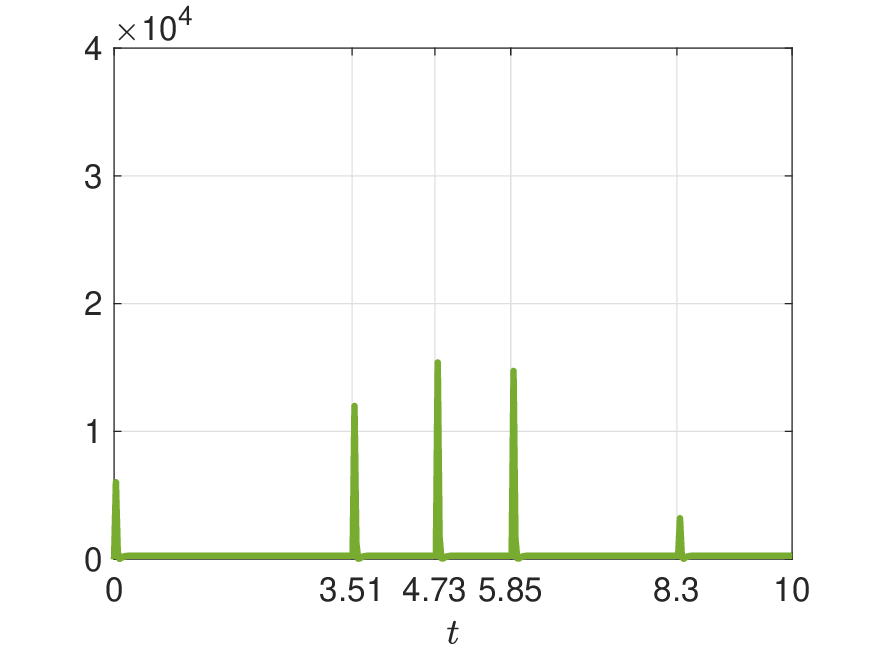} 
\caption{Test 1: $\|y(t)-y_d(t)\|_{L^2(\Omega)}$ (top) and $\|u(t)\|_{L^2(\Omega)}^2$ (bottom) over time with $\bar{T}=0.3$ and $N=12$ following the adaptive (left) and equidistant (right) approach.}\label{fig:control}
\end{figure}

\subsection{Test 2: Solution with a layer at $t=0.5$}
In this numerical test, we consider the optimal control of \eqref{heat} and the cost wants to track a time-dependent reference trajectory. In this example the control horizon will be $[0,1]$, since the quality of our results will not be different if dealing with a larger control horizon. The goal is to well approximate the layer at time $t=0.5$, afterwards the the solution is smooth. The setting for this test example is taken from \cite[Example 5.2]{GonHZ12}, with the following choices: $\nu = 1$ in \eqref{heat} and $\alpha = 1$ in \eqref{costJ}.  The example is built such that the exact optimal solution $(y,u)$ to \eqref{mpcocpreduced} over $[0,1]$ is given by
$$ y(t,x) = \sin (\pi x) \text{atan} ((t-1/2)/\varepsilon), \quad u(t,x) = -\sin (\pi x) \sin (\pi t).$$
The initial condition is $y_\circ (x) = \sin(\pi x) \text{atan}(-1/(2\varepsilon))$. The functions $f$ and $y_d$ are chosen accordingly as 
\begin{align*}
\centering
f(t,x) & = \sin(\pi x) \left( \varepsilon/(t^2 - t + \varepsilon^2 + 1/4) + \pi^2 \text{atan}((t-1/2)/(\varepsilon)) + \sin(\pi t) \right),\\
y_d(t,x)& = \sin(\pi x) \left( \text{atan}((t-1/2)/(\varepsilon)) + \pi \cos(\pi t) - \pi^2 \sin (\pi t) \right).
\end{align*}
For small values of $\varepsilon$ (we use $\varepsilon = 10^{-3}$), the state $y$ develops a very steep gradient at $t = 0.5$, which can be seen in the left panel of Figure~\ref{fig:on_y_MPC}.

We compare the adaptive Algorithm~\ref{Alg:NMPC_on} with a standard equidistant MPC approach. For an exemplary visualization, let us consider the following choices in Algorithm~\ref{Alg:NMPC_on}: $\bar{T}=0.2, N=9$. The numerical state solutions of the controlled problem with the different MPC approaches are shown in the middle and right panel of Figure~\ref{fig:on_y_MPC}. We can see that the standard MPC algorithm with equidistant time grids fails whereas using Algorithm~\ref{Alg:NMPC_on} it is possible to capture the layer at $t=0.5$ and the solution complies much better with the true open-loop state solution over $[0,1]$. 
\begin{figure}[htbp]
\centering
\includegraphics[scale=0.3]{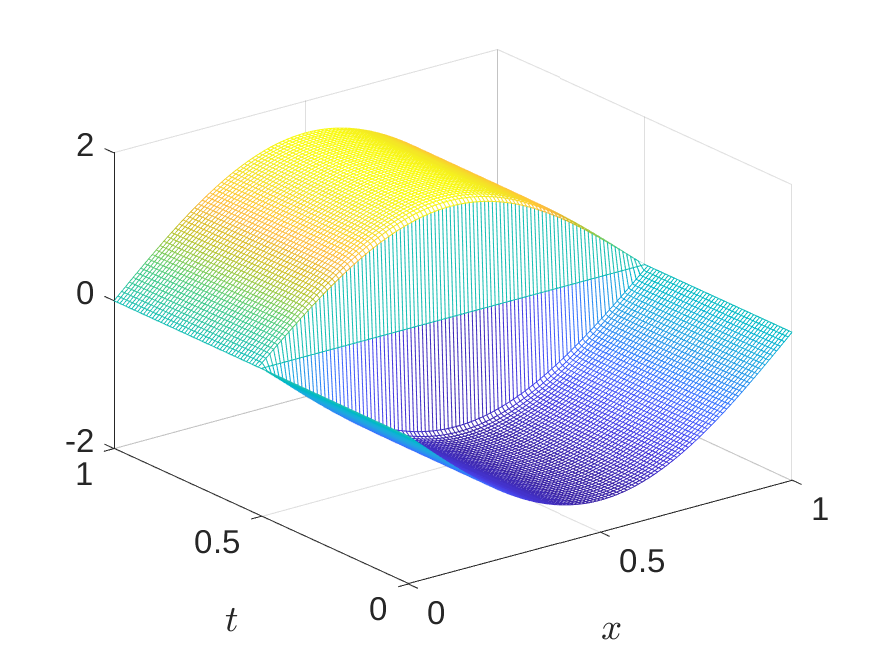}\hspace{-0.4cm}
  \includegraphics[scale=0.3]{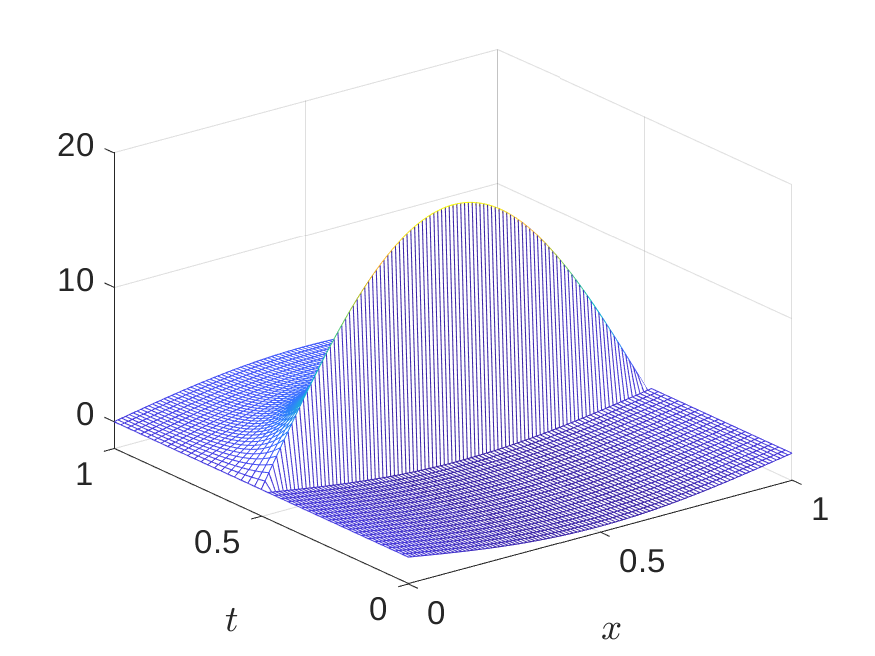} \hspace{-0.4cm}
 \includegraphics[scale=0.3]{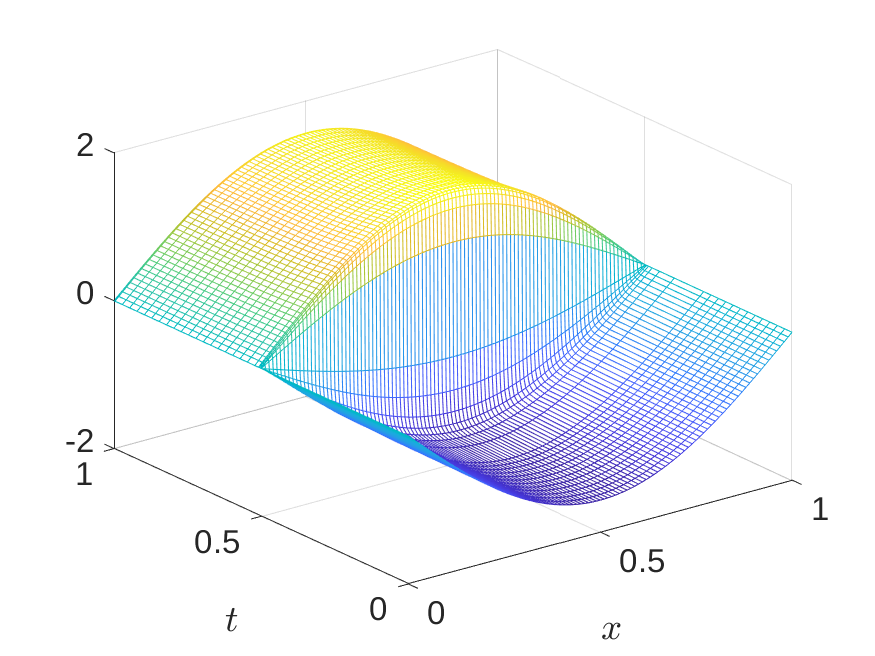} \hspace{-0.4cm}
 \caption{Test 2: True optimal state solution (left), MPC state solution $y$ using a uniform time discretization (middle) and adaptive approach (right).}
 \label{fig:on_y_MPC}
\end{figure}

Let us now provide more details about the temporal grids we obtained with the proposed adaptive scheme. The adaptive grid with a coarse and a fine spatial resolution is shown in the middle and right panel of Figure \ref{fig:onlinegrid}. We observe that the time adaptivity is very insensitive with respect to the spatial resolution, compare Remark~\ref{rem:heuristic}.
\begin{figure}[htbp]
\centering
\includegraphics[scale=0.3]{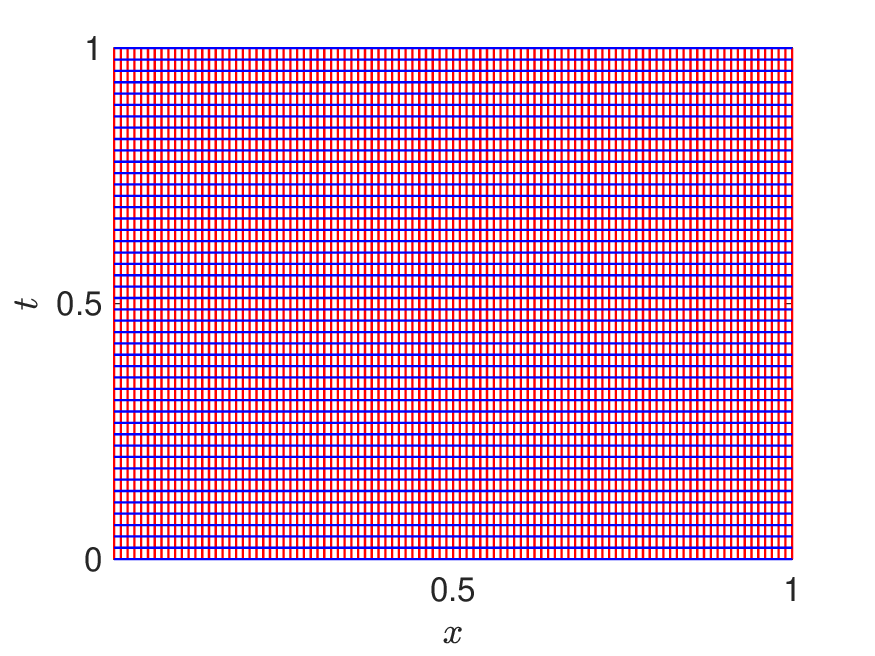}\hspace{-0.4cm}
\includegraphics[scale=0.3]{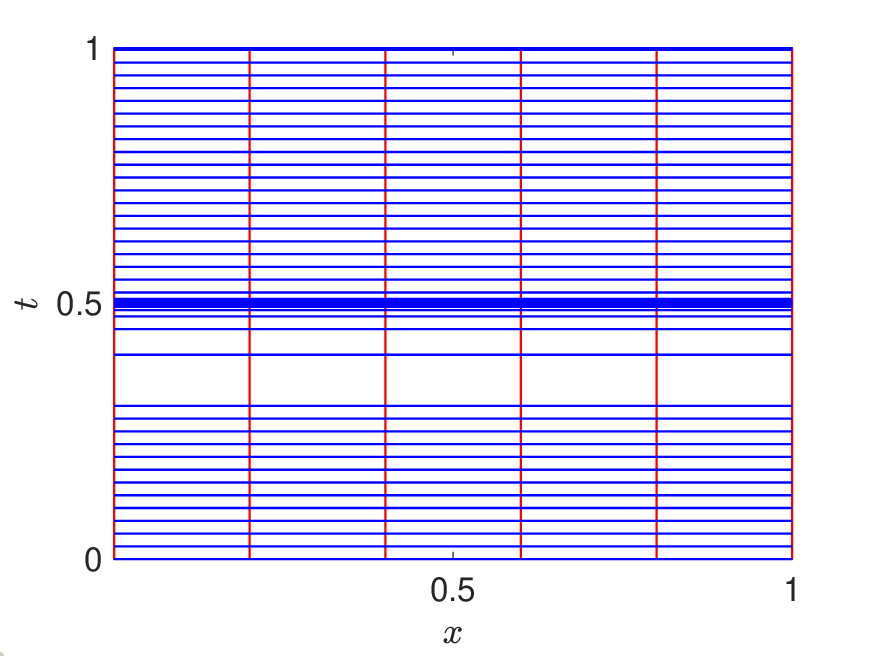}\hspace{-0.4cm}
\includegraphics[scale=0.3]{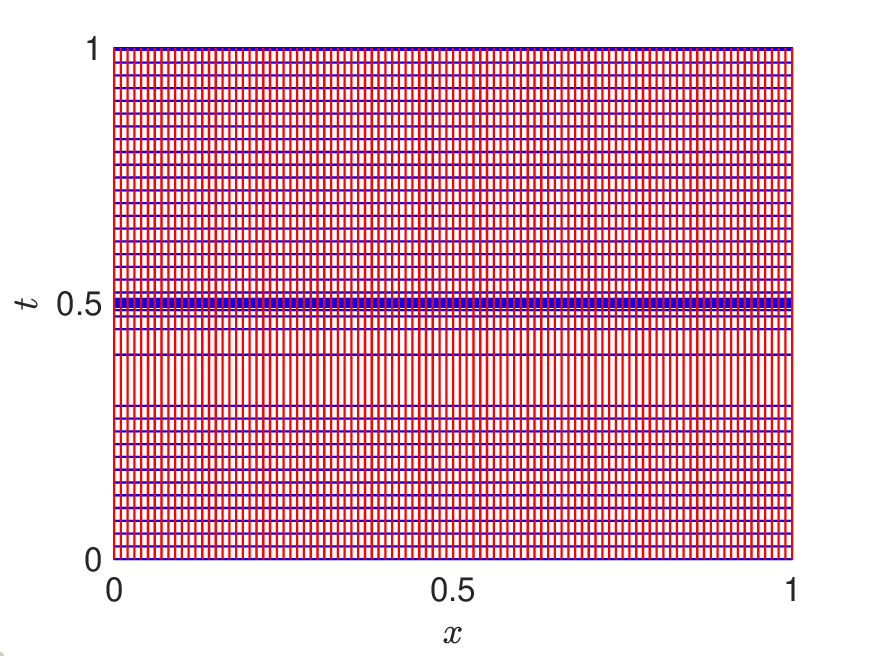}
 \caption{Test 2: Uniform space-time grid with fine spatial resolution (left), adaptive grid with coarse (middle) and fine (right) spatial resolution.}\label{fig:onlinegrid}
\end{figure}
Examples of adaptive time horizons are shown in the top panels of Figure~\ref{fig:onlineadgrids}. As a comparison, the uniform time horizons of the same lengths are shown in the bottom panels of Figure~\ref{fig:onlineadgrids} using the same number of degrees of freedom in each interval. It is clear that the a-posteriori error estimate \eqref{est-thm31} leads to a time grid associated with the open loop optimal state which benefits the accuracy of the control problem.
 
\begin{figure}[htbp]
\centering
\includegraphics[scale=0.3]{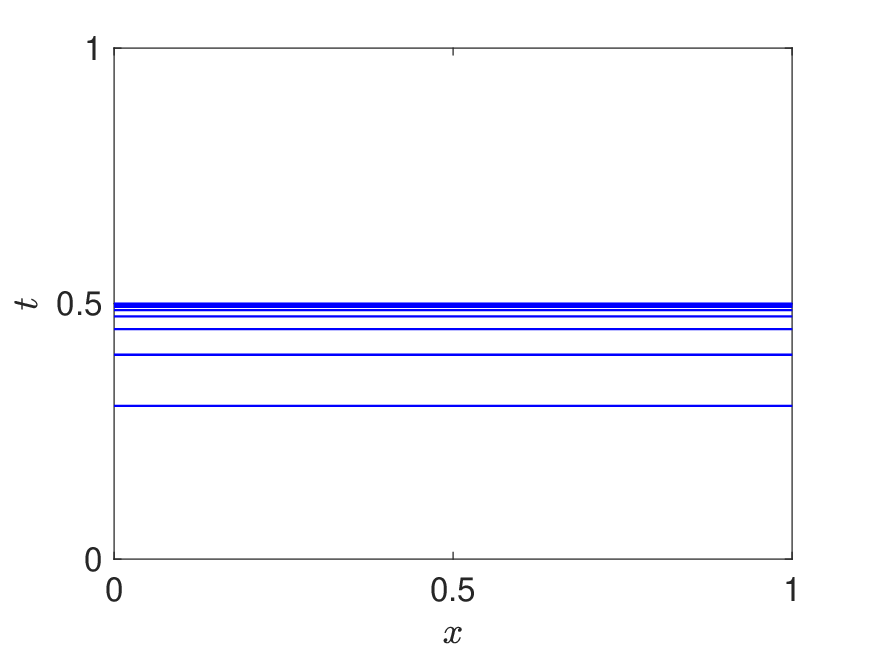}\hspace{-0.4cm}
\includegraphics[scale=0.3]{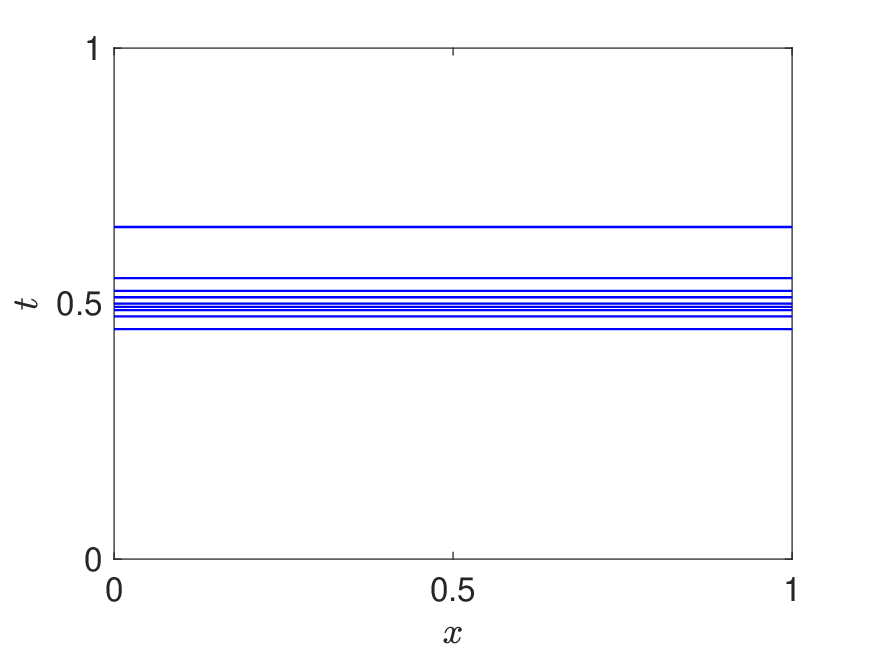}\hspace{-0.4cm}
\includegraphics[scale=0.3]{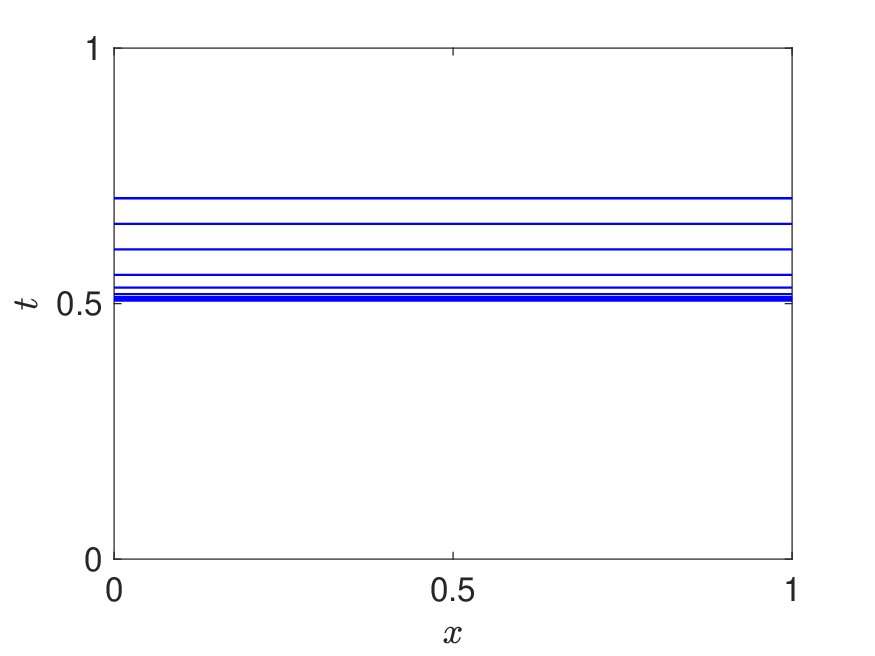}\\
\includegraphics[scale=0.3]{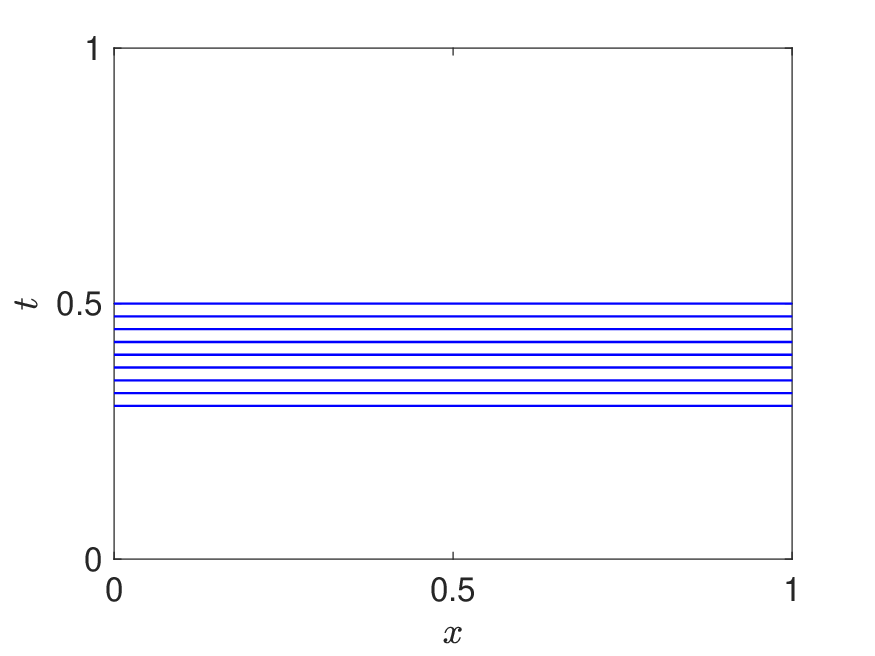}\hspace{-0.4cm}
\includegraphics[scale=0.3]{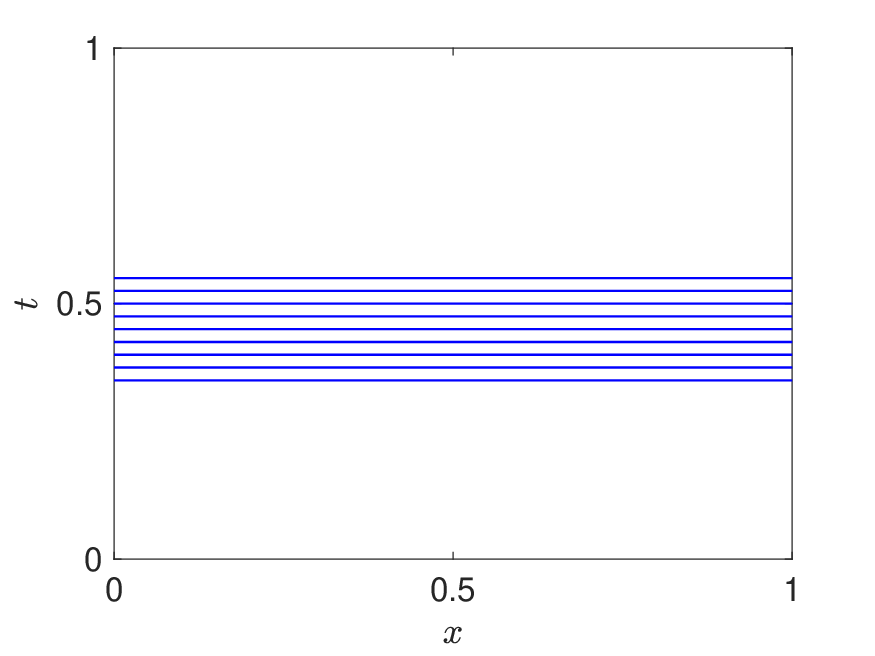}\hspace{-0.4cm}
\includegraphics[scale=0.3]{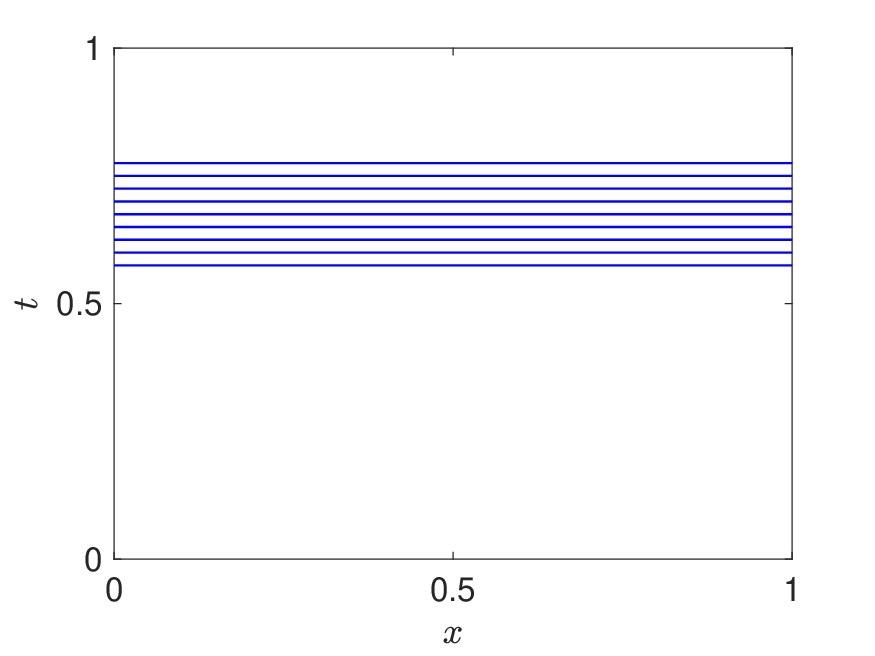}
\caption{Test 2: Adaptive time horizons (top), uniform time horizons according to the standard MPC approach (bottom), MPC iteration levels $i=13$ (left), $i=15$ (middle), $i=24$ (right).}\label{fig:onlineadgrids}
\end{figure}

Finally, we provide an error analysis for the computation of the approximate solutions using an adaptive and an equidistant approach, for different choices of degrees of freedom in time and prediction horizons. For this, we compute the error between the analytical optimal state solution to \eqref{mpcocpreduced} on the finite time domain $[0,T]=[0,1]$ and its numerical approximation using the different MPC approaches measured in the $L^2(0,T;\Omega)-$norm.

In Figure \ref{fig:coston} we compare the $L^2$-error  between the optimal solution to \eqref{mpcocpreduced} and the solutions according to Algorithm~\ref{Alg:NMPC_on}. We fixed the prediction horizon $\bar{T}$ and modified the choice of the instances in each sub interval using the equidistant and adaptive method. As one can see, with this approach we need a small prediction horizon and a large number of time instances to obtain an error of order $10^{-1}$ with an equidistant grid whereas the adaptive method provides a more flexible approach for a small $\bar{T}=\{0.1,0.2,0.3,0.4\}$. Depending on whether the layer at $t=0.5$ is a time discretization point or not, the approximation quality can differ strongly leading to the illustrated zig-zag behavior in the equidistant scheme. Since the exact location of the layer is usually not known a-priorily, an equidistant time grid approach is easy to fail.

  \begin{figure}[htbp]
  \centering
     \includegraphics[scale=0.3]{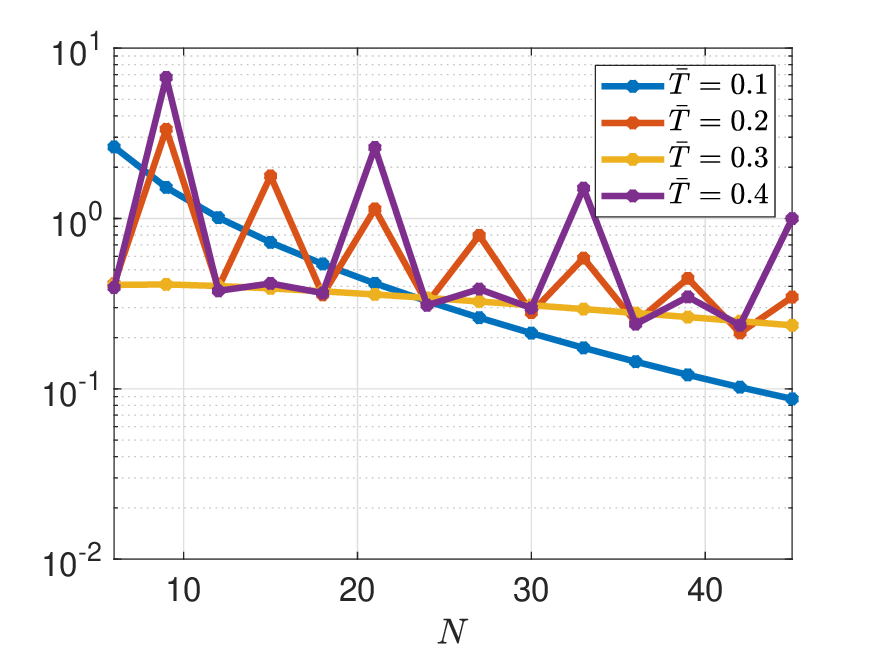}
  \includegraphics[scale=0.3]{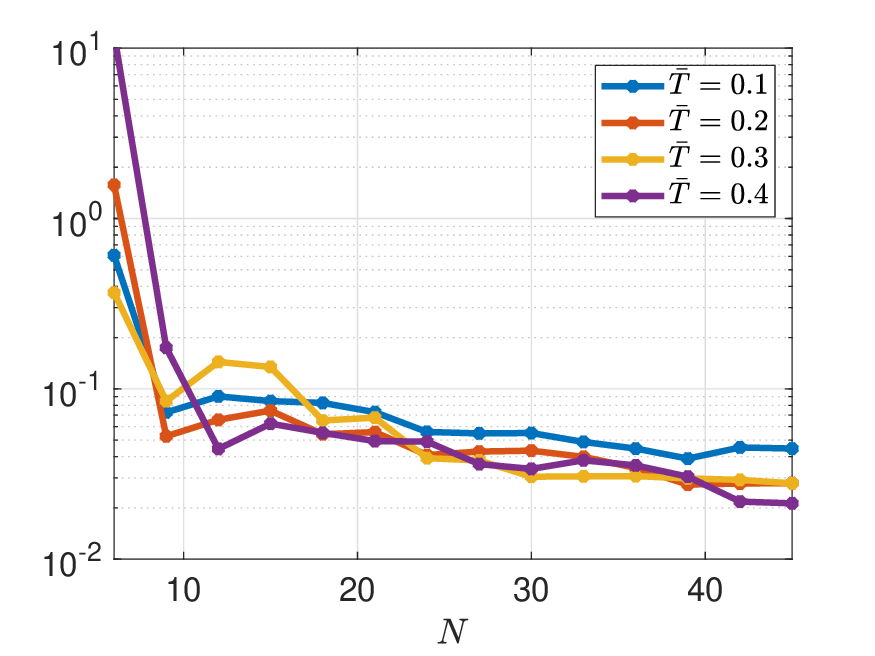}
     \caption{Test 2: $L^2-$error for the MPC approach with equidistant (left) and adaptive (right) time grids.}
  \label{fig:coston}
\end{figure}

In Figure \ref{fig:timeon} we compare the computational time in seconds of the standard MPC Algorithm with Algorithm~\ref{Alg:NMPC_on} including the computational time needed to create the adaptive time discretization within each MPC iteration. Clearly, to obtain a more accurate solution it is computationally more expensive but we also want to remark that the minimum error with the equidistant grid is $0.0872$ computed in $25.87$s whereas, with the adaptive approach, to get an error of $0.0216$ we needed $16.06s$. This shows that our method is more accurate and also more efficient computationally without any a-priori knowledge of the control problem.

  \begin{figure}[htbp]
\centering
     \includegraphics[scale=0.3]{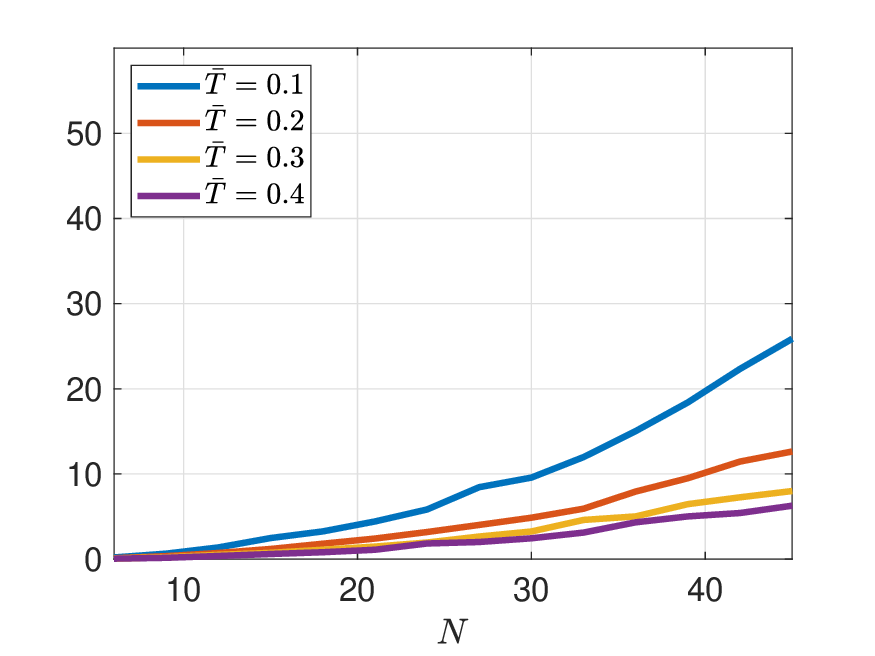}
  \includegraphics[scale=0.3]{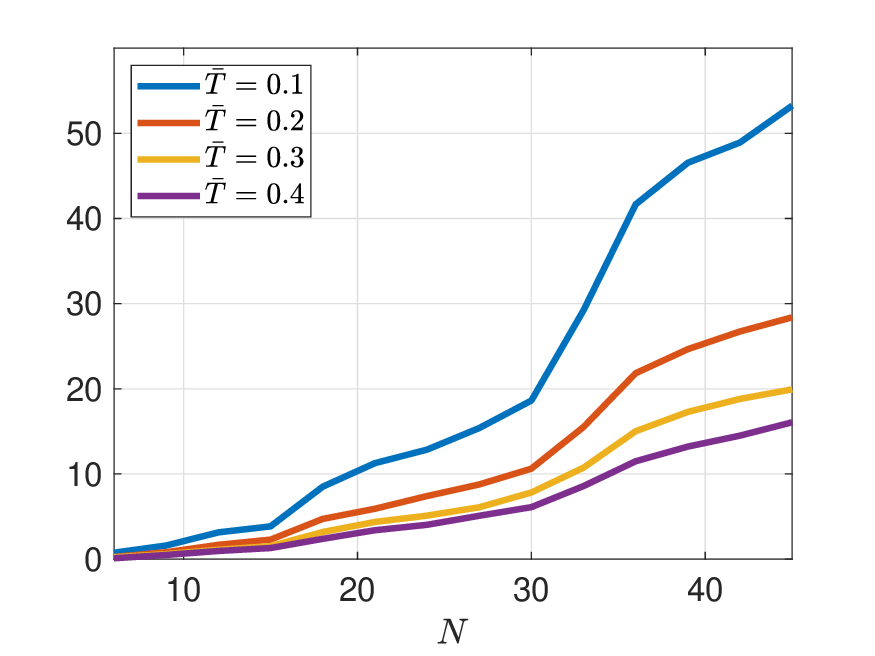}
     \caption{Test 2: {\color{black}Computational} time in seconds for the MPC approach with equidistant (left) and adaptive (right) time grids. }
  \label{fig:timeon}
\end{figure}

\section{Conclusions and Outlook}\label{secconcl}
In this work we have proposed an approach to include time adaptive discretization in the MPC framework. Our approach is fully flexible and relies on a reformulation of the optimal control problem into a second order in time and fourth order in space equation. Our approach does not require further assumptions on the control problem. The use of a-posteriori error estimates to generate the time grid in the MPC method is the important novelty of our work. Numerical tests have shown the efficiency of the method for both accuracy and computational time.
We also want to remark that our approach is particularly suitable when a layer is shown in the solution or the disturbances happen. Other experiments with mild temporal variations did not always show a clear difference between equidistant and adaptive grid. The a-posteriori error indicator delivers an appropriate adaptive time grid even providing a coarse spatial resolution.

In the future, we plan to derive an a-posteriori error estimator for a fully space-time discrete form and to use that indicator for a fully adaptive and automatic MPC scheme, where the idea is to avoid an a-priori choice of the prediction horizon and/or the number of degrees of freedom in each sub-iteration. Another goal is to extend these results to nonlinear control problems and as soon as we increase the dimension of the problem make use of efficient model reduction techniques, such as POD, to decrease the computational time.

\bibliographystyle{abbrv}
\bibliography{references,csc}

\end{document}